\documentclass[12pt]{amsart}
\usepackage{latexsym}
\usepackage{amsmath,amssymb,amsthm,ascmac}
\usepackage{wrapfig}
\usepackage[all]{xy}
\usepackage{proof}

\usepackage[pdftex]{graphicx}

\theoremstyle{plain}
\newtheorem{theorem}{Theorem}[section]
\newtheorem{lemma}[theorem]{Lemma}
\newtheorem{prop}[theorem]{Proposition}
\newtheorem{fact}[theorem]{Fact}
\newtheorem{cor}[theorem]{Corollary}
\newtheorem{obs}[theorem]{Observation}

\newtheorem{question}{Question}

\theoremstyle{definition}
\newtheorem{definition}[theorem]{Definition}
\newtheorem*{ack}{Acknowledgements}
\newtheorem*{claim}{Claim}
\newtheorem{example}[theorem]{Example}

\newtheorem*{remark}{Remark}

\newcommand{\N}{\mathbb{N}}

\newcommand{\upto}{\upharpoonright}
\newcommand{\pcolon}{\colon\!\!\!\subseteq}
\newcommand{\tto}{\rightrightarrows}

\newcommand{\leqLT}{\leq_{LT}}
\newcommand{\geqLT}{\geq_{LT}}
\newcommand{\eqLT}{\equiv_{LT}}
\newcommand{\lLT}{<_{LT}}
\newcommand{\gLT}{>_{LT}}
\newcommand{\leqoLT}{\leq^1_{LT}}

\newcommand{\eqoLT}{\equiv^1_{LT}}
\newcommand{\loLT}{<^1_{LT}}

\newcommand{\dom}{{\rm dom}}

\newcommand{\ep}{\varepsilon}

\newcommand{\fr}{{}^\smallfrown}

\newcommand{\Me}{{\tt Merlin} }
\newcommand{\Mer}{{\tt Merlin}}
\newcommand{\Ar}{{\tt Arthur} }
\newcommand{\Art}{{\tt Arthur}}
\newcommand{\Ni}{{\tt Nimue} }
\newcommand{\Nim}{{\tt Nimue}}

\newcommand{\lrangle}[1]{\langle #1 \rangle}

\newcommand\tboldsymbol[1]{%
\protect\raisebox{0pt}[0pt][0pt]{%
$\underset{\widetilde{}}{\mathbf{#1}}$}\mbox{\hskip 1pt}}

\newcommand{\tpbf}[1]{\tboldsymbol{#1}}

\title
{Lawvere-Tierney topologies for computability theorists}

\author{Takayuki Kihara}

\begin{document}
\maketitle

\begin{abstract}
In this article, we introduce certain kinds of computable reduction games with imperfect information.
One can view such a game as an extension of the notion of Turing reduction, and generalized Weihrauch reduction as well.
Based on the work by Lee and van Oosten, we utilize these games for providing a concrete description of the lattice of the Lawvere-Tierney topologies on the effective topos (equivalently, the subtoposes of the effective topos preordered by geometric inclusion).
As an application, for instance, we show that there exists no minimal Lawvere-Tierney topology which is strictly above the identity topology on the effective topos.
\end{abstract}

\section{Introduction}

\subsection{Summary}\label{sec:summary}

%In this article, we explore a notion of computability-theoretic reduction for certain extended functions.
Our goal in this article is to accomplish a detailed analysis of the entire structure of ``intermediate worlds'' between ``the world of computable mathematics'' and ``the world of set-theoretic mathematics.''
In \cite{Hey82}, Hyland discovered the {\em effective topos} ${\bf Eff}$, and proposed it as the {\em world of computable mathematics}.
In topos theory, there is a notion called a {\em Lawvere-Tierney topology} (also known as a local operator or a geometric modality), and any topology $j$ on a topos $\mathcal{E}$ yields a new subtopos $\mathcal{E}_j\hookrightarrow\mathcal{E}$.
The least topology is the identity topology ${\tt Id}$ that does not cause any change to the base topos.
The largest topology is the indiscrete topology that contracts all truth-values to a single value, and the resulting degenerated topos may be thought of as the {\em world of inconsistent mathematics}.
The next largest topology is the double negation $\neg\neg$.
In the effective topos, the new topos ${\bf Eff}_{\neg\neg}$ created from $\neg\neg$ is exactly the {\em world of set-theoretic mathematics}; that is, ${\bf Eff}_{\neg\neg}\simeq{\bf Set}$.
What this suggests is that analyzing the intermediate topologies between ${\tt Id}$ and $\neg\neg$ on the effective topos may correspond to exploring the intermediate worlds between computable mathematics and set-theoretic mathematics.

Under this perspective, a topology on the effective topos is a kind of data that indicate how much non-computability to add to the world.
In other words, a topology plays the same role as an {\em oracle}.
Indeed, Hyland \cite{Hey82} noticed that each Turing degree $\mathbf{d}$ has a corresponding topology $j_\mathbf{d}$ on the effective topos, which yields the {\em world of $\mathbf{d}$-relatively computable mathematics}.
However, this does not mean that we have exhausted all the topologies, and of course, there may be other topologies besides them.
For instance, instead of a subset of $\N$ or a total function on $\N$, one can use a partial function as an oracle.
Not only that, but even a partial multi-valued function can be used as an oracle, and has a corresponding topology on the effective topos as we observe in this article.
As another example, Pitts \cite{Pit85} found an intermediate topology that is not bounded by any Turing degree topology.
This topology has properties that are far from any of the other topologies mentioned above.
Remarkably, Lee-van Oosten \cite{LvO} gave a concrete presentation of all topologies on the effective topos.

The first step of our work in this article is to capture the presentation of Lee-van Oosten \cite{LvO} within the framework of {\em generalized Weihrauch reducibility} \cite{HiJo16}.
However, generalized Weihrauch reducibility (which involves a perfect information game) itself is insufficient to deal with all topologies, so we introduce an imperfect information game that incorporates some sort of nonuniform computation with advices.
Coincidentally, it turns out that our notion is heavily related to another notion called {\em extended Weihrauch reducibility}, which is introduced in Bauer \cite{Bau21}.
By viewing topologies in this way, it is possible, for example, to position the study of the structure of Lawvere-Tierney topologies as an extension of the Weihrauch-style analogue \cite{pauly-handbook} of reverse mathematics
(note, however, that this is by no means an extension of the standard reverse mathematics \cite{SOSOA:Simpson,Die18} at all;  reverse mathematics has more to do with the internal logic, and more finitary aspects).
By bringing the arguments on topologies into pure computability theory in this way, we solve some problems proposed in \cite{Lee,LvO}.

While the notion of Lawvere-Tierney topology is originally studied in an abstract context, we develop our theory in the most intuitive and elementary way possible.
We believe that it is important for the development of a theory to present it in a way that reduces prior knowledge of the theory as much as possible.
For this purpose, %we avoid the use of tools that keep people away unnecessarily, and 
we maintain appropriate notations and keep the discussion moving forward with concrete ideas.
%Whenever possible, we avoid making meaningless abstractions.
In Section \ref{sec:generalized-reducibility}, we introduce certain kinds of computable reduction games with imperfect information.
By using these games, we also define a notion of computability-theoretic reduction for certain extended functions.
In Section \ref{sec:LT-topologies-chara}, we see that this reducibility notion characterizes the notion of Lawvere-Tierney topology on the effective topos, based on the idea in Lee-van Oosten \cite{LvO}.
In Section \ref{sec:structure-of-LT-topologies}, by using the characterization, we solve some problems on topologies proposed in \cite{Lee,LvO}.
For instance, we see that there is no world of non-computable mathematics which is closest to computable mathematics.
In Section \ref{sec:other-topologies}, we discuss a few other topologies, which has not been studied in the past.
One corresponds to the world of computable mathematics with error probability $\ep$, and the other to computable mathematics with error density $\ep$.

\subsection{Notations}

In this article, we assume that the reader is familiar with elementary facts about computability theory.
For the basics of computability theory, we refer the reader to \cite{Cooper,OdiBook,RogBook,SoareBook}.
%For the basics of (effective) descriptive set theory, we refer the reader to Moschovakis \cite{Mos09}.
%
%For a function $f\colon X\to Y$ and $A\subseteq X$, we use the symbol $f\upto A$ denote the restriction of $f$ up to $A$.
We use the following notations on strings:
Let $\N^{<\N}$ be the set of all finite strings.
For finite strings $\sigma,\tau\in\N^{<\N}$, we write $\sigma\preceq\tau$ if $\sigma$ is an initial segment of $\tau$, and write $\sigma\prec\tau$ if $\sigma$ is a proper initial segment of $\tau$.
We also use the same notation even if $\tau$ is an infinite string, i.e., $\tau\in\N^\N$.
For $\sigma\in\N^{<\N}\cup\N^\N$ and $\ell\in\N$, define $\sigma\upto \ell$ as the initial segment of $\sigma$ of length $\ell$.
For finite strings $\sigma,\tau\in\N^{<\N}$, let $\sigma\fr\tau$ be the concatenation of $\sigma$ and $\tau$.
If $\tau$ is a string of length $1$, i.e., $\tau$ is of the form $\langle n\rangle$ for some $n\in\N$, then $\sigma\fr\langle n\rangle$ is abbreviated to  $\sigma\fr n$.
Similarly, $\langle n\rangle\fr\tau$ is abbreviated as $n\fr\tau$.

A tree is a set $T\subseteq\N^{<\N}$ which is downward closed under $\preceq$.
An element of a tree $T$ is often called a node.
The $\preceq$-least node (i.e., the empty string) of $T$ is called the root, and a $\preceq$-maximal node of $T$ is called a leaf.
We always assume that $\N^\N$ is equipped with the standard Baire topology, that is, the countable product of the discrete topology on $\N$.
For $\sigma\in\N^{<\N}$, let $[\sigma]$ be the clopen set generated by $\sigma$, i.e., $[\sigma]=\{x\in\N^\N:\sigma\prec x\}$.
For $e\in\N$, let $\varphi_e$ be the $e$th partial computable function on $\N$, and $\varphi_e^\alpha$ be the $e$th partial computable function relative to an oracle $\alpha$.
For a partial function $\varphi$, we write $\varphi(n)\downarrow$ if $\varphi(n)$ is defined, and $\varphi(n)\uparrow$ if $\varphi(n)$ is undefined.

As usual, for $n\in\N$, we often use $n$ to denote $\{0,1,\dots,n-1\}$.
We denote a partial function from $X$ to $Y$ as $f\pcolon X\to Y$.
We use the symbol $\mathcal{P}(Y)$ to denote the power set of a set $Y$.
In this article, a partial function $f\pcolon X\to\mathcal{P}(Y)$ is often called a {\em partial multi-valued function} (abbreviated as a {\em multifunction}), and written as $f\pcolon X\tto Y$.
In computable mathematics, we often view a $\forall\exists$-formula $S$ as a partial multifunction.
Informally speaking, a (possibly false) statement $S\equiv\forall x\in X\;[Q(x)\rightarrow\exists yP(x,y)]$ is transformed into a partial multifunction $f_S\pcolon X\rightrightarrows Y$ such that ${\dom}(f_S)=\{x\in X:Q(x)\}$ and $f_S(x)=\{y\in Y:P(x,y)\}$.
Here, we consider formulas as partial multifunctions rather than relations in order to distinguish a {\em hardest} instance $f_S(x)=\emptyset$ (corresponding to a {\em false sentence}) and an {\em easiest} instance $x\in X\setminus{\dom}(f_S)$ (corresponding to a {\em vacuous truth}).
In this sense, a relation does not correspond to a partial multifunction, but to a total multifunction which may take empty value.

\section{Generalized Turing reducibility}\label{sec:generalized-reducibility}

\subsection{Perfect information game}\label{sec:perfect-information-game}
The notion of relative computation (or Turing reducibility) has been first introduced by Turing in 1939.
From that time to the present, its structure has been investigated to an extremely deep level.
As a result, a vast amount of research results are known (see e.g.~\cite{DoHiBook,Handbook,SoareBook} for the tip of the iceberg).
Traditionally, Turing reducibility is usually considered for sets $A\subseteq\N$ or total functions $f\colon\N\to\N$.
However, a slight extension of this, the notion of Turing reducibility for partial functions $f\pcolon\N\to\N$, has also been considered, and the induced structure is known to be isomorphic to the enumeration degrees; see \cite[Section 11.3]{Cooper}.
The concept of relative computability can be extended to even larger classes of functions.
As one such class, we first deal with partial multi-valued functions (abbreviated as multifunctions) on $\N$.
In recent years, the notion of partial multifunction has received a great deal of attention in computability theory and related fields; see e.g.~\cite{pauly-handbook}. 

\medskip

\noindent
{\em Relative Computation Model:}
Let us introduce the notion of computation relative to a partial multifunction on $\N$.
Our computation model is the same as that of an ordinary programming language, except that a program ${\tt P}$ can contain a special instruction of the form ${\tt b:=\verb|[?]|(a)}$.
The computation model accepts a number $n$ and a partial multifunction $f$ as inputs.
The instruction ${\tt b:=\verb|[?]|(a)}$ assigns one of the values of $f({\tt a})$ to the variable ${\tt b}$.
So far, it is exactly the same as an oracle Turing machine.
However, if $f({\tt a})$ is undefined, the computation will never terminate.
Moreover, if $f$ is multi-valued, i.e., if there are more than one possible values for the output of $f({\tt a})$, this generally produces a nondeterministic computation.

\medskip

We write ${\tt P}^f$ for the partial multifunction defined by the above relative computation.
To be precise, for an input $n$, if the program ${\tt P}$ terminates along any path of nondeterministic computation, we declare that $n$ is contained in the domain of ${\tt P}^f$, i.e., ${\tt P}^f(n)\downarrow$.
Furthermore, if the program ${\tt P}$ along some path of nondeterministic computation returns $m$, then we declare $m\in{\tt P}^f(n)$.

\begin{definition}
We say that {\em $g$ is Turing reducible to $f$} (written $g\leq_Tf$) if there exists a program ${\tt P}$ such that ${\tt P}^f$ refines $g$.
Here, for partial multifunctions $g$ and $h$, we say that $h$ refines $g$ if, for any $n$, $n\in{\rm dom}(g)$ implies $n\in{\rm dom}(h)$ and $h(n)\subseteq g(n)$.
\end{definition}

This notion coincides with ordinary Turing reducibility when restricted to total single-valued functions.
One may think that this programming definition is too vague, so we give a mathematically rigorous description of this.
Formally, the process of Turing reduction for partial multifunctions can also be described as a perfect information two-player game.
However, since the players' abilities are asymmetric, we will describe it as a game between \Me and \Art.

\begin{definition}[Perfect information game]
For partial multifunctions $f,g\pcolon\N\tto\N$, let us consider the following perfect information two-player game $\mathfrak{G}(f,g)$:
\[
\begin{array}{rccccccc}
\Me\colon	& x_0	&		& x_1	&		& x_2	&	& \dots \\
\Ar\colon	&		& y_0	&		& y_1	& 		& y_2	& \dots
\end{array}
\]

\noindent
{\it Game rules:}
Each player chooses a natural number at each round.
%Hereafter, we use $[p]$ to denote the partial continuous function on $\N ^\N $ coded by $p$.
Here, \Me and \Ar need to obey the following rules.
\begin{itemize}
\item First, \Me chooses $x_0\in{\dom}(f)$.
\item At the $n$th round, \Ar reacts with $y_n=\langle j,u_n\rangle$.
\begin{itemize}
\item The choice $j=0$ indicates that \Ar makes a new query $u_n$ to $g$.
In this case, we require $u_n\in{\dom}(g)$.
\item The choice $j=1$ indicates that \Ar declares termination of the game with $u_n$.
\end{itemize}
\item At the $(n+1)$th round, \Me responds to the query made by \Ar at the previous stage.
This means that $x_{n+1}\in g(u_n)$.
\end{itemize}

Then, {\em \Ar wins the game $\mathfrak{G}(f,g)$} if either \Me violates the rule before \Ar violates the rule or \Ar obeys the rule and declares termination with $u_n\in f(x_0)$.

\medskip

\noindent
{\em Strategies:}
Hereafter, we require that \Art's moves are chosen in a computable manner.
In other words, \Art's strategy is a code $\tau$ of a partial {\em computable} function $h_\tau\pcolon\N ^{<\N }\to\N $.
On the other hand, \Mer's strategy is any partial function $\sigma\pcolon\N ^{<\N }\to\N $ (which is not necessarily computable).
%Given such strategies $\sigma$ and $\tau$ yield a play $\sigma\otimes\tau$ in the following manner:
%\begin{align*}
%(\sigma\otimes\tau)(0)&=\sigma(\langle\rangle),\\
%(\sigma\otimes\tau)(2n+1)&=h_\tau(\langle(\sigma\otimes\tau)(2m)\rangle_{m\leq n}),\\
%(\sigma\otimes\tau)(2n+2)&=\sigma(\langle(\sigma\otimes\tau)(2m+1)\rangle_{m\leq n}).
%\end{align*}
\Art's strategy $\tau$ is {\em winning} if, as long as \Ar follows the strategy $\tau$, \Ar wins the game, no matter what \Mer's strategy $\sigma$ is.
\end{definition}

\begin{obs}
Let $f$ and $g$ be partial multifunctions.
Then, $f$ is Turing reducible to $g$ if and only if \Ar has a computable winning strategy for $\mathfrak{G}(f,g)$.
\end{obs}

\begin{remark}
A similar notion for partial multifunctions on $\N ^\N $ has been extensively studied, e.g.~in \cite{HiJo16,NePa18,Goh19,westrick2020note,Kih20,DHR20}, and is known as {\em generalized Weihrauch reducibility}.
Indeed, Turing reducibility in the above sense is exactly the restriction of generalized Weihrauch reducibility to functions on $\N $.
\end{remark}

\Art's winning strategy $\tau$ is a {\em one-query strategy} if, for any play following $\tau$, either \Me violates the rule or \Art's second move $y_1$ is of the form $\langle 1,u\rangle$, i.e., $\Art$ declares termination at the second round.

\begin{definition}
Let $f$ and $g$ be partial multifunctions.
We say that {\em $f$ is one-query Turing reducible to $g$} (written $f\leqoLT g$) if there exists \Art's one-query winning strategy $\tau$ for $\mathfrak{G}(f,g)$.
\end{definition}

Equivalently, $f$ is a one-query Turing reducible to $g$ if and only if there exist computable functions $H$ and $K$ such that for any $n$ and $m$,
\[m\in g(H(n))\implies K(n,m)\in f(n).\]

Such an $H$ is called an {\em inner reduction}, and $K$ is called an {\em outer reduction}.

\begin{remark}
A similar notion for partial multifunctions on $\N ^\N $ has been extensively studied under the name {\em Weihrauch reducibility}; see e.g.~Brattka-Gherardi-Pauly \cite{pauly-handbook}.
Indeed, one-query Turing reducibility is exactly the restriction of Weihrauch reducibility to functions on $\N $.
This reducibility is also called many-one reducibility in \cite{Pau17}.
\end{remark}

\begin{remark}
As is well known, it is very difficult to find a natural computably enumerable (c.e.)~set whose Turing degree lies strictly between computable ones and the halting problem; see e.g.~\cite{Mon19}.
As one way to solve this problem of the lack of natural intermediate c.e.~degrees, Simpson \cite{Sim07} proposed to study the Muchnik degrees of $\Pi^0_1$ subsets of Cantor space.
Here, however, we present an alternative solution, which is to consider the Turing degrees of multifunctions on $\N$.
Observe that the Turing degree of a c.e.~set $A\subseteq\N$ is determined by its enumeration time function $\eta_A$, where $\eta_A(n)$ is the stage when $n$ is enumerated into $A$ if such a stage exists; otherwise $\eta_A(n)=0$.
One can easily see that the graph of $\eta_A$ is always co-c.e., i.e., $\Pi^0_1$.
%a total function on $\N$ is computable if and only if its graph is $\Pi^0_1$.
%dealing with the Turing degrees of c.e.~sets is the same as dealing with those of co-c.e.~sets ($\Pi^0_1$ sets), which is identical to the Turing degrees of two-valued functions with $\Pi^0_1$ graphs [{\bf NOT TRUE}].
In this light, one can consider that the counterpart of the Turing degrees of c.e.~sets in the multi-valued context is the Turing degrees of multifunctions with $\Pi^0_1$ graphs.
In Example \ref{exa:LLPO}, we give a natural intermediate $\Pi^0_1$ degree between computable problems and the halting problem.

To point out the relevance of the $\Pi^0_1$ multifunctions on $\N$ to the $\Pi^0_1$ subsets of $\N^\N$, note that if $f\colon\N\tto\N$ is a $\Pi^0_1$ multifunction, then the product $\prod_{n\in\N}f(n)$ forms a $\Pi^0_1$ subset of $\N^\N$.
However, be careful about that Turing reducibility for multifunctions is entirely different from Muchnik reducibility for their product sets.
\end{remark}

\begin{example}[Intermediate Turing degree]\label{exa:LLPO}
The following is the $\N$-version of a well-studied principle, called the {\em lesser limited principle of omniscience}.
\begin{align*}
{\rm dom}({\tt LLPO})&=\{e\in\N:|\{j<2:\varphi_e(j)\downarrow\}|\leq 1\},\\
{\tt LLPO}(e)&=\{0,1\}\setminus \{j<2:\varphi_e(j)\downarrow\}.
\end{align*}

There are a huge number of mathematical principles which are equivalent to ${\tt LLPO}$; see Diener \cite{Die18} and also Brattka-Gherardi-Pauly \cite{pauly-handbook}.
The principle ${\tt LLPO}$ may also be called {\em de Morgan's law for $\Sigma^0_1$ formulas}.
In the realizability context, this is closely related to Lifschitz realizability \cite{Lif,vOBook}.
It is not hard to see that the Turing degree of ${\tt LLPO}$ strictly lies between the computable problems and the halting problem.
This also follows from our results in later sections (see Propositions \ref{prop:llpo-vs-error-pos} and \ref{prop:llpo-error}).

Note that ${\tt LLPO}$ is one-query Turing equivalent to a multifunction with a $\Pi^0_1$ graph.
Given $e\in\N$, define $\psi_e$ as follows:
\begin{align*}
\psi_e(0)\downarrow&\iff(\exists s\in\N)\;[\varphi_e(0)[s]\downarrow\;\land\;(\forall t<s)\;\varphi_e(1)[t]\uparrow],\\
\psi_e(1)\downarrow&\iff(\exists s\in\N)\;[\varphi_e(1)[s]\downarrow\;\land\;\varphi_e(0)[s]\uparrow],
\end{align*}
where $\varphi_e(j)[s]$ is the stage $s$ approximation of $\varphi_e(j)$.
Note that it is not possible for both $\psi_e(0)$ and $\psi_e(1)$ to terminate; that is, we always have $|\{j<2:\psi_e(j)\downarrow\}|\leq 1$.
Then we define ${\tt L}(e)=\{0,1\}\setminus\{j<2:\psi_e(j)\downarrow\}$.
Obviously, the graph of the multifunction ${\tt L}\colon\N\tto\N$ is $\Pi^0_1$, and ${\tt L}\eqoLT{\tt LLPO}$.
\end{example}

\subsection{Imperfect information game}\label{sec:imperfect-information-game}

There are various forms of computation, one of which is the notion of probabilistic computation.
As a simple example, let us consider the situation where a program ${\tt P}$ is given an oracle $\alpha$ at random, and for an input $n$, the oracle computation ${\tt P}^\alpha(n)$ halts with probability at least $1-\ep$.
In other words, this is the situation where
\[\mu(A)\geq 1-\ep\;\land\;(\forall \alpha\in A)\;{\tt P}^\alpha(n)\downarrow,\]
for some set $A\subseteq 2^\N$.
Here, $\mu$ is the uniform probability measure on $2^\N$ (i.e., the probability measure by infinite fair coin flips).
This probabilistic computation yields a multifunction such that the value ${\tt P}^\alpha(n)$ for each $\alpha\in A$ is a possible output.
This computation has two parameters, $n$ and $A$.
Of course, $n$ is an input given by us, while $A$ is a witness that the computation halts except for probability at most $\ep$.
It is only guaranteed that such an $A$ exists mathematically, but the computer does not know what exactly $A$ is.

Let us write ${\tt ProbError}_\ep{\tt P}$ for the procedure of giving an oracle randomly to the program ${\tt P}$ and having it perform a computation with error probability at most $\ep$.
If one wants to make explicit a parameter $A$ which witnesses that the computation succeeds with error probability at most $\ep$ for an input $n$, we use the following notation:
\[{\tt ProbError}_\ep{\tt P}(n\mid A)\]

A pair $(n\mid A)$ of parameters is properly accepted only if $A$ witness that ${\tt P}^\alpha(n)$ halts except for probability at most $\ep$:
\[
{\tt ProbError}_\ep{\tt P}(n\mid A)\downarrow\iff A\subseteq 2^\N\;\land\;\mu(A)\geq 1-\ep\;\land\;(\forall \alpha\in A)\;{\tt P}^\alpha(n)\downarrow.
\]

Then, the value ${\tt P}^\alpha(n)$ for each $\alpha\in A$ is a possible output:
\[
y\in{\tt ProbError}_\ep{\tt P}(n\mid A)\iff \exists \alpha\in A\;[{\tt P}^\alpha(n)=y].
\]

Although the roles of $n$ and $A$ are entirely different, if we just treat them formally, the above process can be regarded as a partial multifunction
\[{\tt ProbError}_\ep{\tt P}\pcolon\N\times\mathcal{P}(2^\N)\tto\N.\]

In this sense, both $n$ and $A$ can be thought of as inputs for the above multifunction, but $n$ is an input that is disclosed during the computation, while $A$ is an unknown input that cannot be accessed during the computation.
Hence, we call $n$ a public input, and $A$ a secret input.

%In general, let us consider a partial multifunction $g\pcolon\N\times\Lambda\tto\N$, where $\Lambda$ is an arbitrary set.
%In practice, however, it is more convenient to think that this is a partial function $g\pcolon\N\to[\subseteq\Lambda\tto\N]$ which takes a value in the set $[\subseteq\Lambda\tto\N]$ of partial multifunctions from $\Lambda$ to $\N$.

\begin{definition}\label{def:LT-problem}
A partial multifunction $g\pcolon\N\times\Lambda\tto\N$, where $\Lambda$ is a set, is called a {\em bilayer function} in this article (any suggestions for a better name for this notion would be welcome).
In this context, a pair $(n,c)\in\N\times\Lambda$ is always written as $(n\mid c)$.
For $(n\mid c)\in{\rm dom}(g)$, we call $n$ a {\em public input} and $c$ a {\em secret input}.
Then, the public domain ${\rm dom}_{\rm pub}(g)$ of $g$ is defined as the set of all $n\in\N$ such that $(n\mid c)\in{\rm dom}(g)$ for some $c\in\Lambda$.
\end{definition}

\begin{example}\label{exa:pmv-is-bilayer}
Any partial multifunction $g\pcolon\N\tto\N$ can be identified with the following bilayer function $\hat{g}\pcolon\N\times\{\ast\}\tto \N$:
\[\hat{g}(n\mid \ast)=g(n).\]
\end{example}

\begin{remark}
If one wants to avoid dealing with an arbitrary set $\Lambda$, one can just consider $G(n)=\{g(n\mid c):c\in\Lambda\mbox{ and }(n\mid c)\in{\rm dom}(g)\}$, which yields $G\pcolon\N\to\mathcal{P}\mathcal{P}(\N)$.
Conversely, if a partial function of the form $G\pcolon\N\to\mathcal{P}\mathcal{P}(\N)$ is given, one can always assume that the elements of $G(n)$ are indexed as $G(n)=\{p^n_c:c\in\Lambda_n\}$.
Then, we consider $g(n\mid c)$ to mean $p^n_c$, which yields a partial multifunction $g\pcolon \N\times\Lambda\tto\N$.

Indeed, previous studies, such as Lee-van Oosten \cite{LvO}, rather deal only with $\mathcal{P}\mathcal{P}(\N)$-valued functions.
In the terms of Bauer \cite{Bau21}, a $\mathcal{P}\mathcal{P}(\N)$-valued function is called an {\em extended Weihrauch degree}, and a partial multifunction as seen as an extended Weihrauch degree is called a {\em modest extended Weihrauch degree}.
However, from the point of view of advised computation, there are advantages to the way of looking at it as in Definition \ref{def:LT-problem}.
\end{remark}

Let us consider relative computation with a bilayer function oracle.
In our computation model, a secret input for an oracle acts like an advice string in computational complexity theory.
For the role of advice in computability theory, we refer the reader to Brattka-Pauly \cite{BrPa10} and Ziegler \cite{Zie12}.
One-query computation with advice in the context of $\N^\N$-computation has also been discussed there.

\begin{example}\label{exa:adviceN}
In the context of $\N^\N$-computation, the bilayer function ${\tt Advice}_\N\colon\{\ast\}\times\N\to\N$ defined by ${\tt Advice}_\N(\ast\mid n)=n$ can be used to deal with {\em nonuniform computability} \cite{BrPa10,Zie12}.
However, in the context of $\N$-computation, ${\tt Advice}_\N$ is too strong and produce the $\neg\neg$-topology on the effective topos \cite{LvO}.
Several variants of random advice in the context of $\N^\N$-computation have also been studied in \cite{BrPa10,BGH15}.
\end{example}

\noindent
{\em Relative Computation Model:}
%Let us consider what an $f$-relative computation is for a bilayer function $f$.
Our computation model deals not only with one-query relative computation, but also with many-query relative computation.
During a computation with a bilayer function oracle $f$, when the program makes a query $n$ to $f$, the advisor chooses a parameter $c$.
However, the information of $c$ chosen by the advisor is not given to the machine, but only the information of one of the possible values of $f(n\mid c)$ is given.
If this process computes a partial multifunction $g$ when the advisor secretly makes the best choice, then we declare that {\em $g$ is Turing reducible to $f$ in the bilayered context}, and write $g\leqLT f$.

\medskip

Again, one may think that this programming definition is too vague, so we give a mathematically rigorous description of this.
Formally, this procedure can be understood by describing it as an imperfect information game between three players, \Mer, \Art, and \Nim.
The player \Me makes a public input $x_0$ and a secret input $c_0$ on his first move.
Here, among the moves of \Mer, only the secret input $c_0$ is invisible to \Art.
All of \Nim's moves are visible to \Mer, but not to \Art, a mere human being.
The players \Me and \Nim, who are not mere humans, can see all the previous moves at each round.

\begin{definition}[Imperfect information game]
For bilayer functions $f$ and $g$, let us consider the following imperfect information three-player game $\mathfrak{G}(f,g)$:
\[
\begin{array}{rcccccccccc}
\Me\colon	& (x_0\mid c_0)	&		&		& x_1	&		&		& x_2	&		& 		& \dots \\
\Ar\colon	&				& y_0	&		&		& y_1	& 		&		& y_2	& 		& \dots \\
\Ni\colon		&				&		& z_0	&		&		& z_1	& 		&		& z_2	& \dots
\end{array}
\]

%Each player chooses a natural number at each round.
%Hereafter, we use $[p]$ to denote the partial continuous function on $\N ^\N $ coded by $p$.

\noindent
{\it Game rules:}
Here, the players need to obey the following rules.
\begin{itemize}
\item First, \Me chooses a pair $(x_0\mid c_0)\in{\rm dom}(f)$.
\item At the $n$th round, \Ar reacts with $y_n=\langle j,u_n\rangle$.
\begin{itemize}
\item The choice $j=0$ indicates that \Ar makes a new query $u_n$ to $g$.
In this case, we require $u_n\in{\dom}_{\rm pub}(g)$.
\item The choice $j=1$ indicates that \Ar declares termination of the game with $u_n$.
\end{itemize}
\item At the $n$th round, \Ni makes an advice parameter $z_n$ such that $(u_n\mid z_n)\in{\rm dom}(g)$.
\item At the $(n+1)$th round, \Me responds to the query made by \Ar and \Ni at the previous stage.
This means that $x_{n+1}\in g(u_n\mid z_n)$.
\end{itemize}

Then, {\em \Ar and \Ni win the game $\mathfrak{G}(f,g)$} if either \Me violates the rule before \Ar or \Ni violates the rule, or both \Ar and \Ni obey the rule and \Ar declares termination with $u_n\in f(x_0\mid c_0)$.

\medskip

\noindent
{\it Strategies:}
As noted above, \Ar can only read the moves $x_0,x_1,x_2,\dots$, and the other players can see all the moves.
Moreover, we require that \Art's moves are chosen in a computable manner.
In other words, \Art's strategy is a code $\tau$ of a partial {\em computable} function $h_\tau\pcolon\N ^{<\N }\to\N$, which reads \Mer's moves $x_0,\dots,x_n$ and then returns $y_n$.
On the other hand, \Me and \Nim's strategies are any partial functions (which are not necessarily computable).

A pair $(\tau\mid \eta)$ of \Art's computable strategy $\tau$ and \Nim's strategy $\eta$ is called an \Art-\Ni strategy.
An \Art-\Ni strategy $(\tau\mid\eta)$ is {\em winning} if, as long as \Ar and \Ni follow the strategy $(\tau\mid\eta)$, \Ar and \Ni win the game, no matter what \Mer's strategy $\sigma$ is.
\end{definition}

We now introduce a generalization of Turing reducibility for bilayer functions.

\begin{definition}\label{def:LT-reducibility}
Let $f$ and $g$ be bilayer functions.
We say that {\em $f$ is bilayered Turing reducible (or LT-reducible) to $g$} (written $f\leqLT g$) if there exists a winning \Art-\Ni strategy for $\mathfrak{G}(f,g)$.
\end{definition}

Obviously, bilayered Turing reducibility for partial multifunctions (which can be viewed as bilayer functions as in Example \ref{exa:pmv-is-bilayer}) is the same as Turing reducibility.

\begin{remark}
The notion of an \Art-\Ni strategy is strongly related to the notion of a {\em dedicated sight} in Lee-van Oosten \cite[Definition 4.3]{LvO}.
The statement that $S$ is a $(z,\theta,p)$-dedicated sight roughly corresponds to that $(z\mid S)$ is a winning \Art-\Ni strategy witnessing $\dot{p}\leqLT \theta$, where $\dot{p}(\ast\mid\ast)=p$ for $p\subseteq\N$.
\end{remark}

Before examining bilayered Turing reducibility, we again consider one-query reductions.
A winning \Art-\Ni strategy $(\tau\mid\eta)$ is a {\em one-query strategy} if, for any play following $(\tau\mid\eta)$, either \Me violates the rule or \Art's second move $y_1$ is of the form $\langle 1,u\rangle$, i.e., $\Art$ declares termination at the second round.

\begin{definition}
Let $f$ and $g$ be bilayer functions.
We say that {\em $f$ is one-query bilayered Turing reducible to $g$} (written $f\leqoLT g$) if there exists a one-query winning \Art-\Ni strategy $(\tau\mid\eta)$ for $\mathfrak{G}(f,g)$.
\end{definition}

Equivalently, $f$ is a one-query bilayered Turing reducible to $g$ if and only if there exist computable functions $H$ and $K$ and a function $L$ such that for any $(n\mid c)$ and $m$,
\[m\in g(H(n)\mid L(n,c))\implies K(n,m)\in f(n\mid c).\]

Such an $H$ is called an {\em inner reduction}, and $K$ is called an {\em outer reduction}.
We also call $L$ a {\em secret inner reduction}.
%The one-query Turing reduction is called many-one reducibility in Pauly (2017), and 

\begin{remark}
One-query bilayered Turing reducibility for bilayer functions (extended Weihrauch degrees) is simply called {\em Weihrauch reducibility} in Bauer \cite{Bau21}.
The algebraic structure of the one-query bilayered Turing degrees (the extended Weihrauch degrees) has been studied there.
\end{remark}

\begin{prop}\label{prop:oqTuring-is-preorder}
$\leqoLT$ is a preorder.
\end{prop}

\begin{proof}
Reflexivity is trivial.
For transitivity, let $\langle H_0,K_0,L_0\rangle$ witness $f\leqoLT g$ and $\langle H_1,K_1,L_1\rangle$ witness $g\leqoLT h$.
Then $m\in h(H_1\circ H_0(n)\mid L_1(H_0(n),\eta_0(n,c)))$ implies $K_1(H_0(n),m)\in g(H_0(n)\mid L_0(n,c))$, which implies $K_0(n,K_1(H_0(n),m))\in f(n\mid c)$.
%Consider the compositions $H=H_1\circ H_0$, $K=K_0\circ K_1$, $\eta=\eta_1\circ\eta_0$.
Hence, $H_1\circ H_0$ is an inner reduction, $(n,m)\mapsto K_0(n,K_1(H_0(n),m))$ is an outer reduction, and $(n,c)\mapsto L_1(H_0(n),L_0(n,c))$ is a secret inner reduction witnessing $f\leqoLT h$.
\end{proof}

\begin{prop}\label{prop:Turing-is-preorder}
$\leqLT$ is a preorder.
\end{prop}

\begin{proof}
Reflexivity is trivial.
For transitivity, we only need to combine the argument in Proposition \ref{prop:oqTuring-is-preorder} and the proof of transitivity of generalized Weihrauch reducibility \cite[Proposition 4.4]{HiJo16}.
We assume that $f\leqLT g$ and $g\leqLT h$.
To avoid confusion, we name $\Ar_0$, $\Me_0$, and $\Ni_0$ for the players in the game $\mathfrak{G}(f,g)$, and $\Ar_1$, $\Me_1$, and $\Ni_1$ for the players in the game $\mathfrak{G}(g,h)$.
Let $(\tau_i\mid\eta_i)$ be a winning $\Ar_i$-$\Ni_i$ strategy for the corresponding game for each $i<2$.
In the following, we assume that $\Ar_i$ and $\Ni_i$ always follow their winning strategies.
We construct a winning \Art-\Ni strategy for $\mathfrak{G}(f,h)$.

Let $(x\mid c)$ be \Mer's first move in the game $\mathfrak{G}(f,h)$.
Then, consider $(x\mid c)$ as $\Mer_0$'s first move in the game $\mathfrak{G}(f,g)$ as well, and simulate a play following the $\Art_0$-$\Ni_0$ strategy $(\tau_0\mid\eta_0)$.
Along such a play, if $\Ar_0$ declares termination with some $u$ at some round, then $\Ar$ also declares termination with the same value $u$.
If $\Ar_0$ and $\Ni_0$ make a query $(u\mid z)$ to $g$ at some round, then think of $(u\mid z)$ as $\Me_1$'s first move in the game $\mathfrak{G}(g,h)$, and simulate a play following the $\Art_1$-$\Ni_1$ strategy $(\tau_1\mid\eta_1)$.
During this subplay, \Ar and \Ni simply copy the moves made by $\Ar_1$ and $\Ni_1$, respectively, and play them as their own moves.
Here, \Me copies $\Me_1$'s moves, except for the first move, and uses them directly in his own moves.
Since $\Ar_1$ and $\Ni_1$ follow their winning strategies, $\Ar_1$ declares termination with some $v$ at some round, and moreover $v\in g(u\mid z)$.
Hence, one can think of such $v$ as $\Me_0$'s response to the previous move $(u\mid z)$ by $\Ar_0$ and $\Ni_0$.
This allows the game $\mathfrak{G}(f,g)$ to move on to the next round (and this can be simulated by \Ar and \Nim, since they know the value of $v$).
Repeating this process, since $\Ar_0$ and $\Ni_0$ follow their winning strategies, $\Ar_0$ declares termination with some $w$ at some round, and moreover $w\in f(x\mid c)$.
Therefore, $\Ar_0$ also declares termination with some $w\in f(x\mid c)$ at some round.
Hence, the \Art-\Ni strategy described above is winning for $\mathfrak{G}(f,h)$, which concludes $f\leqLT h$.
\end{proof}

Note that the rule of the game $\mathfrak{G}(f,g)$ does not mention $f$ except for Player I's first move.
Hence, if we skip Player I's first move, we can judge if a given play follows the rule without specifying $f$.
Such a restricted game is denoted by $\mathfrak{G}(g)$.
 {\em \Ar and \Ni win the game $\mathfrak{G}(g)$} if either \Me violates the rule before \Ar or \Ni violates the rule, or both \Ar and \Ni obey the rule and \Ar declares termination.

\begin{definition}
Given a bilayer function $h$, let us define the new bilayer function $h^\Game$ as follows:
An input for $h^\Game$ is an \Art-\Ni strategy $(\tau\mid\eta)$, where \Art's strategy $\tau$ is a public input, and \Nim's strategy $\eta$ is a secret input.
\begin{itemize}
\item $h^\Game(\tau\mid\eta)$ is defined only if, along any play following the strategy $(\tau\mid\eta)$, \Ar and \Ni win the game $\mathfrak{G}(h)$ whatever \Mer's strategy is.
\item $u\in h^\Game(\tau\mid\eta)$ if and only if there is a play in $\mathfrak{G}(h)$ that follows the strategy $(\tau\mid\eta)$ such that \Ar declares termination with $u$ at some round, where all players obey the rule.
\end{itemize}
\end{definition}

The first condition says that $(\tau\mid\eta)\in h^\Game$ if and only if $(\tau\mid\eta)$ is a winning \Art-\Ni strategy for $\mathfrak{G}(h)$ in a certain sense, and in particular, \Ar declares termination at some round unless \Me violates the rule.
Alternatively, $h^\Game$ can be thought of as a universal machine for $h$-relative computation.

\begin{prop}\label{prop:game-closure-reduction}
For bilayer functions $g$ and $h$,  $g\leqLT h$ if and only if $g\leqoLT h^\Game$.
\end{prop}

\begin{proof}
$(\Rightarrow)$
Let $(\tau\mid\eta)$ be a winning \Art-\Ni strategy witnessing $g\leqLT h$.
Given an input $(n\mid c)$ for $g$, define $\tau_n(\sigma)=\tau(n\fr\sigma)$ and $\eta_{n,c}(\sigma)=\eta(\langle n,c\rangle\fr\sigma)$.
Any $u\in h^\Game(\tau_n\mid\eta_{n,c})$ corresponds to a play in $\mathfrak{G}(g,h)$ following the strategy $(\tau\mid\eta)$, where \Mer's first move is $(n\mid c)$, and $\Ar$ declares termination with $u$.
Since $(\tau\mid\eta)$ is winning in $\mathfrak{G}(g,h)$, we must have $u\in g(n\mid c)$.
Thus, $n\mapsto\tau_n$ is an inner reduction, $(n,c)\mapsto\eta_{n,c}$ is a secret inner reduction, and $(n,u)\mapsto u$ is an outer reduction witnessing $g\leqoLT h^\Game$.

$(\Leftarrow)$
Let $\langle H,K,L\rangle$ witness $g\leqoLT h^\Game$.
As $H$ is an inner reduction, note that $H(n)$ is (a code of) \Art's strategy, and think of $H(n)(\sigma)$ as \Art's move after reading \Mer's moves $\sigma$.
Then, define $\tau(n\fr\sigma)=H(n)(\sigma)$ if $H(n)(\sigma)$ does not declare termination, i.e., $H(n)(\sigma)$ is of the form $\langle 0,u\rangle$.
If $H(n)(\sigma)$ declares termination with $u$, then define $\tau(n\fr\sigma)=\langle 1,K(n,u)\rangle$; that is, $\tau(n\fr\sigma)$ declares termination with $K(n,u)$.
Clearly, $\tau$ is computable.
We also define $\eta(\langle n,c\rangle\fr\sigma)=L(n,c)(\sigma)$.

Assume that \Ar and \Ni follow the strategy $(\tau\mid\eta)$ in the game $\mathfrak{G}(g,h)$.
If \Mer's first move is $(n\mid c)$, then, by the definitions of $\tau$ and $\eta$, \Ar and \Ni follow the strategy $(H(n)\mid L(n,c))$ in the subgame $\mathfrak{G}(h)$ until \Ar declares termination.
Since $(H(n)\mid L(n,c))\in{\rm dom}(h^\Game)$, either \Me violates the rule before \Ar or \Ni violates the rule, or all players obey the rule and \Ar declares termination.
Assume that \Me follows a strategy obeying the rule.
Then, in the subgame $\mathfrak{G}(h)$, \Ar declares termination with $u$ at some round, i.e., $u\in h^\Game(H(n)\mid L(n,c))$.
Hence, by the definition of $\tau$, in the game $\mathfrak{G}(g,h)$, \Ar declares termination with $K(n,u)$.
As $\langle H,K,L\rangle$ are reductions witnessing $g\leqoLT h^\Game$, $u\in h^\Game(H(n)\mid L(n,c))$ implies $K(n,u)\in g(n\mid c)$.
This verifies that $(\tau\mid\eta)$ is a winning \Art-\Ni strategy witnessing $g\leqLT h$.
\end{proof}

\begin{prop}\label{prop:game-closure-twice}
$h^{\Game\Game}\eqoLT h^\Game$.
\end{prop}

\begin{proof}
Obviously, $h^\Game\leqoLT h^{\Game\Game}$.
For the other direction, by the reflexivity of $\leqoLT $, we have $h^{\Game\Game}\leqoLT h^{\Game\Game}$, which implies that $h^{\Game\Game}\leq_{T}h^{\Game}$ by Proposition \ref{prop:game-closure-reduction}.
Similarly, we also have $h^\Game\leqLT h$.
Since $\leqLT $ is transitive by Proposition \ref{prop:Turing-is-preorder}, we have $h^{\Game\Game}\leqLT h$.
Hence, we get $h^{\Game\Game}\leqoLT h^\Game$ by Proposition \ref{prop:game-closure-reduction}.
\end{proof}

\begin{remark}
In the context of $\N^\N$-computability, the closure operator $-^\Game$ restricted to partial multifunctions is essentially the same as the diamond operator in \cite{NePa18,westrick2020note}.
\end{remark}

\begin{remark}
Several variants of Weihrauch reducibility can be explained by using bilayer functions in the context of $\N^\N$-computation.
For instance, $f$ is computable reducible to $g$ in the sense of \cite{Dzh16,HiJo16} if and only if $f$ is one-query bilayered Turing reducible to $(g\mid{\tt Advice}_\N)$ (see Definition \ref{def:mul-plus-mono}) if we properly extend the above notions to the context of $\N^\N$-computability.
The notion of omniscient computable/Weihrauch reducibility \cite{MoPa19,DHR20,DzPa20} can also be explained in the bilayer context.
\end{remark}

The bilayered Turing degrees of concrete bilayer functions are examined in Sections \ref{sec:structure-of-LT-topologies} and \ref{sec:other-topologies}.

\section{Lawvere-Tierney topology}\label{sec:LT-topologies-chara}

\subsection{Realizability}

For sets $p,q\subseteq\mathcal{P}(\N)$ we define $p\leq q$ if there exists a partial computable function $\varphi\pcolon\N\to\N$ such that $n\in p$ implies $\varphi(n)\in q$.
Then $(\mathcal{P}(\N),\leq)$ forms a Heyting algebra, where the Heyting operations are given as follows:
\begin{align*}
p\land q&=\{\langle m,n\rangle:m\in p\;\land\;n\in q\},\\
p\lor q&=\{\langle 0,n\rangle:n\in p\}\cup\{\langle 1,n\rangle:n\in q\},\\
p\rightarrow q&=\{e:(\forall n\in\N)\;[n\in p\;\to\;\varphi_e(n)\in q]\}.
\end{align*}

Now we put $\Omega=\mathcal{P}(\mathbb{N})$, and consider $\Omega$ as the set of truth values in the world of computable mathematics.
Indeed, the morphism tracked by $true\colon 1\to\Omega$, where $true(\ast)=\N$, is a subobject classifier in the effective topos.
For more information on the effective topos, see \cite{Hey82,Lee,vOBook}.

Let $A$ be a propositional formula, where all propositional variables belong to $\Omega$.
Then $A$ can be thought of as an element of $\Omega$ by using Heyting operations on $\Omega$.
We say that {\em $e$ realizes $A$} if $e$ belongs to $A$ under the above interpretation.
We also say that {\em $e$ realizes $\forall p\;A(p)$} if $e\in A(p)$ for any $p\in\Omega$.
Then, {\em $A$ is realizable} if some $e\in\N$ realizes $A$. 
%The notion of computability yields
%Define a partial binary operation $\ast$ on $\N$ by $e\ast n=m$ if $\varphi_e(n)\downarrow=m$; and $e\ast n$ is undefined if $\varphi_e(n)\uparrow$.
%
%For instance, the category of modest sets (i.e., represented spaces) and 
%The category of assemblies over Kleene's first algebra is a regular category, and its ex/reg completion is called the effective topos.

\subsection{Lawvere-Tierney topologies}\label{sec:LT-topologies}

In this section, we reveal the hidden relationship between bilayered Turing degrees and Lawvere-Tierney topologies (Definition \ref{def:LT-topology}) on the effective topos.
As mentioned in Section \ref{sec:summary}, we regard an Lawvere-Tierney topology on the effective topos as a kind of data that indicate how much non-computability to add to the world, and thus, a topology plays the same role as an oracle.
In this regard, Hyland \cite{Hey82} found an embedding of the Turing degrees (of total single-valued functions on $\N$) into the lattice of Lawvere-Tierney topologies on the effective topos.
Hyland's embedding can be extended to partial functions or even partial multifunctions on $\N$.
By extending this further, we show that there exists an isomorphism between the bilayered Turing degrees and the lattice of Lawvere-Tierney topologies on the effective topos (Corollary \ref{cor:Turing-vs-LT}).
This guarantees that, in the strict sense, any topology on the effective topos can be identified with a (bilayer) oracle.
Note that most of the results in Section \ref{sec:LT-topologies} have almost been proven by Lee and van Oosten \cite{LvO}, although their language is completely different from ours, and in particular they do not give any computational interpretation of their notions.

A function $\beta\colon\Omega\to\Omega$ is said to be {\em computably monotone} if the following is realizable:
\[\forall p,q\;[(p\to q)\to(\beta(p)\to \beta(q))].\]

In other words, there exists $e$ such that for any $p,q\in\Omega$ if $a$ realizes $p\to q$ then $\varphi_e(a)$ realizes $\beta(p)\to \beta(q)$.
We define a preorder on computably monotone functions on $\Omega$ as follows:
\[\alpha\leq_{\rm rea}\beta\iff\mbox{``$\forall p\;[\alpha(p)\to \beta(p)]$'' is realizable},\]
that is, there exists $e$ such that, for any $p\in\Omega$, $e$ realizes $\alpha(p)\to \beta(p)$.
For a preorder, its quotient by the induced equivalence relation is called the poset reflection.

\begin{theorem}\label{thm:one-quary-vs-monotone}
The poset reflections of the following preorders are isomorphic:
\begin{itemize}
\item The one-query bilayered Turing preorder $\leqoLT $ on bilayer functions.
\item The preorder $\leq_{\rm rea}$ on computably monotone functions on $\Omega$.
\end{itemize}
\end{theorem}

\begin{proof}
Given a bilayer function $g$, we define a function $g^\to\colon\Omega\to\Omega$ as follows:
\[\langle n,e\rangle\in g^\to(p)\iff\mbox{$e$ realizes $g(n\mid c)\to p$ for some $c$.}\]

Roughly speaking, $g^\to(p)$ is a problem that asks us to solve a problem $p$ with the help of $g$.
Of the solutions $n$ and $e$ to $g^\to(p)$, we sometimes call $n$ an {\em inner reduction} and $e$ an {\em outer reduction}.
Indeed, if we put $\dot{p}(\ast\mid\ast)=p$, then $\langle n,e\rangle \in g^\to(p)$ if and only if $\langle n,e\rangle$ witnesses $\dot{p}\leqoLT g$.
Note that, if $\theta$ is a bilayer function, $\theta^\to$ is essentially the same as $G_\theta$ under the notation in Lee-van Oosten \cite[page 873]{LvO}.
One can easily see that $g^\to$ is computably monotone (see also \cite{LvO}).
We first show the following:
\[g\leqoLT h\iff g^\to\leq_{\rm rea} h^\to.\]

For the forward direction, assume that $\langle H,K,L\rangle$ witnesses $g\leqoLT h$, and $\langle n,e\rangle$ witnesses $\dot{p}\leqoLT g$ with some $c$.
Then, the composition $H(n)$ of inner reductions and the composition $\varphi_e\circ K$ of outer reductions witness $\dot{p}\leqoLT h$ with $L(n,c)$.
This is because for any solution $y\in h(H(n)\mid L(n,c))$ we have $K(n,y)\in g(n\mid c)$ as $\langle H,K,L\rangle$ is a reduction triple, and for any $z\in g(n\mid c)$ we have $\varphi_e(z)\in\dot{p}(\ast\mid\ast)=p$; hence $\varphi_e\circ K(n,y)\in p$.
Put $n'=H(n)$ and let $e'_n$ be an index of the computable function $y\mapsto \varphi_e\circ K(n,y)$.
Then we have $\langle n',e'_n\rangle\in h^\to(p)$.
Clearly $\langle n,e\rangle\mapsto\langle n',e'_n\rangle$ is computable and independent of $p$.
Hence we get $g^\to\leq_{\rm rea}h^\to$.

For the backward direction, let $e$ be a realizer for $g^\to(p)\to h^\to(p)$ for any $p\in\Omega$.
Given $(n\mid c)$, let us consider $p=g(n\mid c)$.
It is obvious that $\langle n,{\rm id}\rangle$ witnesses $\dot{p}\leqoLT g$.
Thus, $\varphi_e(n,{\rm id})=\langle m_n,d_n\rangle$ witnesses $\dot{p}\leqoLT h$.
In other words, $y\in h(m_n\mid c')$ implies $\varphi_{d_n}(y)\in\dot{p}(\ast\mid\ast)=g(n\mid c)$ for some $c'$.
One can find an index $d$ such that $\varphi_d(n,x)=\varphi_{d_n}(x)$.
Then, $n\mapsto m_n$ is an inner reduction, $c\mapsto c'$ is a secret inner reduction, and $\varphi_d$ is an outer reduction for $g\leqoLT h$.

Now, given a computably monotone function $\beta\colon\Omega\to\Omega$, define a bilayer function $\beta^\leftarrow$ as follows:
\[{\rm dom}(\beta^\leftarrow)=\{(n\mid c):n\in \beta(c)\},\qquad \beta^\leftarrow(n\mid c)=c.\]

Note that $\beta\mapsto \beta^{\leftarrow}$ is essentially the same as the transformation $f\mapsto\theta$ in the proof of Lee-van Oosten \cite[Theorem 2.4]{LvO}.
Therefore, as in the proof of \cite[Theorem 2.4]{LvO} (that shows $f\equiv F(f)\equiv G_\theta$ under their terminology), one can see that $(\beta^{\leftarrow})^{\rightarrow}\equiv_{\rm rea}\beta$.
This concludes the proof.
\end{proof}

\begin{remark}
Recall that, in \cite{Bau21}, one-query bilayered Turing reducibility is called extended Weihrauch reducibility.
Together with the result in Bauer \cite{Bau21} that extended Weihrauch degrees and instance degrees (over a relative partial combinatory algebra) are equivalent preorders, one can deduce that the orders in Theorem \ref{thm:one-quary-vs-monotone} are also isomorphic to instance degrees over Kleene's first algebra.
\end{remark}

By the proof of Theorem \ref{thm:one-quary-vs-monotone}, note that we also have $(g^\to)^\leftarrow\eqoLT g$ and
\[g\leq_{\rm rea}h\iff g^{\leftarrow}\leqoLT h^{\leftarrow}.\]

We next consider the notion of Lawvere-Tierney topology (also known as local operator or geometric modality), which is, in general, defined as a certain operator on the truth-value object in a given topos.
In the effective topos, it is essentially the same as the following notion (see also \cite{Lee,LvO}):

\begin{definition}\label{def:LT-topology}
A function $j\colon\Omega\to\Omega$ is a {\em Lawvere-Tierney topology} if all of the following are realizable
\begin{enumerate}
\item $\forall p\;[p\to j(p)]$.
\item $\forall p\;[j(p\land q)\leftrightarrow j(p)\land j(q)]$.
\item $\forall p\; [j(j(p))\to j(p)]$
\end{enumerate}
\end{definition}

Recall from the proof of Theorem \ref{thm:one-quary-vs-monotone} that $g^\to(p)$ is the set of reduction pairs for $\dot{p}\leqoLT g$ (where a secret reduction is not included).
Therefore, by Proposition \ref{prop:game-closure-reduction}, $g^{\Game\to}(p)$ is essentially the set of \Art's winning strategies for $\dot{p}\leqLT g$.
We next see that the function $g^{\Game\to}\colon\Omega\to\Omega$ is always a Lawvere-Tierney topology.

\begin{obs}
Let $h$ be a bilayer function.
Then, $h^{\Game\to}\colon\Omega\to\Omega$ is a Lawvere-Tierney topology.
\end{obs}

\begin{proof}
(1)
If one can solve a problem $p$ without any help, it is clear that one can also solve the problem $p$ with the help of $h^\Game$.
(2)
For the backward direction, if one can solve problems $p$ and $q$ with the help of $h^\Game$, then by running these strategies in parallel, one can also solve the problem $p\land q$ with the help of $h^\Game$.
The forward direction is obvious.

(3)
By definition, $\langle \tau,e\rangle\in h^{\Game\to}h^{\Game\to}(p)$ if and only if $e$ realizes $h^\Game(\tau\mid\eta)\to h^{\Game\to}(p)$ for some $\eta$.
Thus, if $u\in h^\Game(\tau\mid\eta)$ and $\varphi_e(u)=\langle \tau'(u),e'(u)\rangle$ then $e'(u)$ realizes $h^\Game(\tau'(u)\mid\eta'(u))\to p$ for some $\eta'(u)$.
As in the proof of Proposition \ref{prop:Turing-is-preorder}, we combine two games, but this time in series.
On the first game $\mathfrak{G}(h)$, \Ar and \Ni follow their strategies $(\tau\mid\eta)$ with one exception:
Even if the strategy $\tau$ declares termination $u$, then \Ar do not declare termination, but move on to the next game which is also $\mathfrak{G}(h)$.
Then \Ar and \Ni next follow their strategies $(\tau'(u)\mid\eta'(u))$ with one exception:
If the strategy $\tau$ declares termination $v$, then \Ar declares termination with $\varphi_{e'(u)}(v)$.
Note that if \Me obeys the rule, then we always have $\varphi_{e'(u)}(v)\in p$.
This can be viewed as a single game which is also $\mathfrak{G}(h)$, and we write $(\tau''\mid\eta'')$ for the \Art-\Ni strategy described above.
It is easy to check that $(\tau''\mid\eta'')\in{\rm dom}(h^\Game)$, and if $w\in h^\Game(\tau''\mid\eta'')$ then $w\in p$.
Therefore, the identity map realizes $h^\Game(\tau''\mid\eta'')\to p$.
Hence, $\langle \tau'',{\rm id}\rangle\in h^{\Game\to}(p)$.
Clearly, $\tau\mapsto\tau''$ is computable, and thus an index of the computable function $\langle \tau,e\rangle\mapsto\langle \tau'',{\rm id}\rangle$ realizes $h^{\Game\to}h^{\Game\to}(p)\to h^{\Game\to}(p)$ for all $p$.
\end{proof}

For any monotone function $\beta\colon\Omega\to\Omega$, consider $L(\beta)=\beta^{\leftarrow\Game\to}$.
By the above observation, $L(\beta)$ is always a topology.

\begin{theorem}[see also {\cite[Proposition 1.2]{LvO}}]
Let $\beta\colon\Omega\to\Omega$ be a computably monotone function.
Then, $L(\beta)$ is the $\leq_{\rm rea}$-least topology such that $L(\beta)\geq_{\rm rea}\beta$.
\end{theorem}

\begin{proof}
First, since $\beta^\leftarrow\leqoLT  \beta^{\leftarrow\Game}$, we have $\beta\equiv_{\rm rea} \beta^{\leftarrow\to}\leq_{\rm rea} \beta^{\leftarrow\Game\to}=L(\beta)$ by Theorem \ref{thm:one-quary-vs-monotone}; that is, $\beta\leq_{\rm rea} L(\beta)$ always holds.
Thus, it remains to show that $\beta\leq_{\rm rea}j$ implies $L(\beta)\leq_{\rm rea}j$ for any topology $j$.
To prove this, as in \cite[Proposition 1.2]{LvO}, consider the following:
\[L'(\beta)(p):=\forall q\ [[(p\to q)\land (\beta(q)\to q)]\to q].\]

As shown in \cite[Proposition 1.2]{LvO}, $\beta\leq_{\rm rea}j$ implies $L'(\beta)\leq_{\rm rea} j$.
Hence, it remains to show $L(\beta)\leq_{\rm rea}L'(\beta)$.
%We show that $L(f)(p)\to L'(f)(p)$ is realizable.
Here, recall that a realizer for $L(\beta)(p)=\beta^{\leftarrow\Game\rightarrow}(p)$ is a pair $\langle{\tt d},{\tt e}\rangle$ of an inner reduction ${\tt d}$ and an outer reduction ${\tt e}$ for $\dot{p}\leqoLT \beta^{\leftarrow\Game}$, i.e., ${\tt e}$ realizes $\beta^{\leftarrow\Game}({\tt d}\mid c)\to p$ for some $c$.
In this case, we must have $({\tt d}\mid c)\in{\rm dom}(\beta^{\leftarrow\Game})$, which means that $({\tt d}\mid c)$ is a winning \Art-\Ni strategy for the game $\mathfrak{G}(\beta^\leftarrow)$.

To compute a realizer of $L'(\beta)(p)$, assume that we are given a realizer ${\tt a}$ of $p\to q$ and a realizer ${\tt b}$ of $\beta(q)\to q$ in the premise of $L'(\beta)(p)$, which are independent of $q$.
On some play of the game $\mathfrak{G}(\beta^\leftarrow)$, if $(n\mid z)$ is \Ar and \Nim's queries to $\beta^{\leftarrow}$ in their moves (without declaring termination) at some round, then we have $(n\mid z)\in{\rm dom}(\beta^\leftarrow)$, which means that $n\in \beta(z)$ and $\beta^\leftarrow(n\mid z)=z$.
Since ${\tt b}$ realizes $\beta(z)\to z$, we have ${\tt b}\cdot n\in \beta^{\leftarrow}(n\mid z)$.
Hence, ${\tt b}$ yields \Mer's strategy which obeys the rule.
Therefore, one can simulate one of the plays of the game $\mathfrak{G}(\beta^\leftarrow)$ from the information in {\tt d}, c, and {\tt b}.
Since $({\tt d}\mid c)$ is a winning \Art-\Ni strategy, and \Mer's strategy {\tt b} obeys the rule, \Ar declares termination at some round along this play.
In particular, one can compute \Art's final move in this play, which yields some $m\in \beta^{\leftarrow\Game}({\tt d}\mid c)$.
By applying {\tt e} to this result, we can get a realizer for $p$.
Furthermore, by applying {\tt a} to this result, we get a realizer for $q$.
This procedure yields a realizer for $L(\beta)(p)\to L'(\beta)(p)$ independent of $p$, and thus, $L(\beta)\to L'(\beta)$ is realizable.
\end{proof}

As in \cite{LvO}, we define 
\[\alpha\leq_{\rm L}\beta\iff \alpha\leq_{\rm rea}L(\beta).\]

In summary, for bilayer functions $f$ and $g$, we obtain
\[f\leqLT g\iff f\leqoLT g^\Game\iff f^{\to}\leq_{\rm rea}g^{\Game\to}=L(g^\to)\iff f^\to\leq_{\rm L}g^\to.\]

Here, the first equivalence follows from Proposition \ref{prop:game-closure-reduction}, and the second one follows from Theorem \ref{thm:one-quary-vs-monotone}.
As any computably monotone function is $\equiv_{\rm rea}$-equivalent to a function of the form $f^\to$, this concludes the following:

\begin{cor}\label{cor:Turing-vs-LT}
The poset reflections of the following preorders are isomorphic:
\begin{itemize}
\item The bilayered Turing preorder $\leqLT$ on bilayer functions.
\item The preorder $\leq_{\rm L}$ on computably monotone functions on $\Omega$.
\item The preorder $\leq_{\rm rea}$ on Lawvere-Tierney topologies.
\end{itemize}
\end{cor}

In summary, for any bilayer function $g$, the map $g^{\Game\rightarrow}$ which, given a problem $p$, returns a problem asking us to giving an \Art's winning strategy for $\dot{p}\leqLT g$ is always a Lawvere-Tierney topology, and conversely, every Lawvere-Tierney topology on the effective topos can be described in this way.

\section{On the structures of Lawvere-Tierney topologies}\label{sec:structure-of-LT-topologies}

\subsection{Turing degrees of choices from co-$m$-tons}

If the domain (for public inputs) of a bilayer function $g$ is a singleton (i.e., an input is always of the form $(\ast\mid c)$), then we call $g$ a {\em basic bilayer function}.
In this section, we deal with the basic bilayer function ${\tt Error}_{m/k}$ for $m<k\in\N$ defined by
\begin{align*}
{\rm dom}({\tt Error}_{m/k})&=\{(\ast\mid A):A\subseteq\{0,\dots,k-1\}\;\land\;|A|=m\},\\
{\tt Error}_{m/k}(\ast\mid A)&=\{0,\dots,k-1\}\setminus A
\end{align*}

This is a problem such that $m$ of the $k$ choices are wrong.
In particular, one-query ${\tt Error}_{m/k}$-relative computation is the one in which $k$ computations are run in parallel, $m$ of which may be wrong.
Note that the basic bilayer function ${\tt Error}_{m/k}$ is denoted as $\mathcal{O}^k_m$ in \cite{LvO}.
Lee-van Oosten \cite{LvO} proposed to study the structure of $(\{{\tt Error}_{m/k}^\to\}_{m<k},\leq_{\rm rea},\leq_{\rm L})$.
By Theorem \ref{thm:one-quary-vs-monotone} and Corollary \ref{cor:Turing-vs-LT}, this is the same as the examining the structure of $(\{{\tt Error}_{m/k}\}_{m<k},\leqoLT ,\leqLT)$.
%For instance, it is obvious that $m\leq n$ and $k\leq \ell$ imply that ${\tt Error}_{m/k}<_{1T}{\tt Error}_{m+1/k}$
%Moreover, Lee-van Oosten showed that, for any $1<2m<k$, ${\tt Error}_{m/k}<_{1T}{\tt Error}_{m+1/k}$, and moreover ${\tt Error}_{m/k}<_{T}{\tt Error}_{m+1/k}$ whenever $\lceil\frac{k}{m+1}\rceil<\lceil\frac{k}{m}\rceil$.

Of course, the basic bilayer function ${\tt Error}_{m/k}$ is closely related to the well-known notion, the lessor limited principle of omniscience (recall Example \ref{exa:LLPO}).
Here, we consider its generalization:
\begin{align*}
{\rm dom}({\tt LLPO}_{m/k})&=\{e\in\N:|\{j<k:\varphi_e(j)\downarrow\}|\leq m\},\\
{\tt LLPO}_{m/k}(e)&=\{0,\dots,k-1\}\setminus \{j<k:\varphi_e(j)\downarrow\}.
\end{align*}

Clearly, ${\tt LLPO}$ is equivalent to ${\tt LLPO}_{1/2}$.
The principle ${\tt LLPO}_{1/\ell}$ is first introduced in Richman \cite{Ri90}, and extensively studied in constructive mathematics and related areas.
For the computability-theoretic study, see Brattka-Gherardi-Pauly \cite{pauly-handbook}.
Note that one can deduce several results on the structure of $(\{{\tt LLPO}_{m/k}\}_{m<k},\leqoLT )$ from the work by Cenzer-Hinman \cite{CH08}.

One may relativize ${\tt LLPO}_{m/k}$ by replacing $\varphi_e$ with $\varphi_e^\alpha$ for a given oracle $\alpha\in 2^\N$, and then the resulting function is denoted by ${\tt LLPO}^\alpha_{m/k}$.
Recall from Example \ref{exa:pmv-is-bilayer} that a partial multifunction can be thought of as a bilayer function.
The following is obvious:

\begin{obs}\label{obs:llpo-from-error-pos}
For any oracle $\alpha$, ${\tt LLPO}^\alpha_{m/k}\leqoLT {\tt Error}_{m/k}$.
\end{obs}

If $\alpha=\emptyset$, we can do a little better.

\begin{prop}\label{prop:llpo-vs-error-pos}
${\tt LLPO}_{m/k}\leqoLT {\tt Error}_{m/k+1}$.
\end{prop}

\begin{proof}
We define a secret inner reduction $L$ as follows:
For any $e\in{\rm dom}({\tt LLPO}_{m/k})$,
\[
L(e)=
\begin{cases}
\{j<k\colon\varphi_e(j)\downarrow\}&\mbox{ if }|\{j<k\colon\varphi_e(j)\downarrow\}|=m\\
\{j<k\colon\varphi_e(j)\downarrow\}\cup\{k\}&\mbox{ if }|\{j<k\colon\varphi_e(j)\downarrow\}|<m
\end{cases}
\]

One can easily check that $(\ast\mid L(e))$ belongs to the domain of ${\tt Error}_{m/k+1}$.
For an outer reduction $K$, define $K(e,j)=j$ for any $j<k$.
To compute $K(e,k)$, wait for finding $m$ many $j<k$ such that $\varphi_e(j)\downarrow$.
If it is found at some stage, then $K(e,k)$ is defined as the least $\ell<k$ such that $j\not=\ell$ for any such $j<k$, so $K(e,k)\downarrow=\ell$ implies $\varphi_e(j)\uparrow$.
Otherwise, the computation never terminates, i.e., $K(e,k)\uparrow$.

We claim that $\langle L,K\rangle$ witnesses ${\tt LLPO}_{m/k}\leqoLT {\tt Error}_{m/k+1}$.
Assume $a\in{\tt Error}_{m/k+1}(\ast\mid L(e))$.
If $a<k$ then we have $K(e,a)=a\not\in\{j<k\colon\varphi_e(j)\downarrow\}$; hence $K(e,a)\in{\tt LLPO}_{m/k}$.
If $a=k$ then by our definition of $L(e)$, we must have $|\{j<k\colon\varphi_e(j)\downarrow\}|=m$.
The computation for $K(e,a)$ eventually recognizes this fact at some stage, and this implies that $K(e,a)\not\in\{j<k\colon\varphi_e(j)\downarrow\}$; hence $K(e,a)\in{\tt LLPO}_{m/k}$.
\end{proof}

In particular, we get ${\tt LLPO}\leqoLT {\tt Error}_{1/3}$.
One can easily see that the above proof indeed shows that ${\tt LLPO}_{m/k}\leqoLT {\tt LLPO}^{\emptyset'}_{m/k+1}$.
However, it cannot be improved any further.
To prove this, we need a little preparation.
We say that a tree $T\subseteq\N^{<\N}$ is {\em $n$-fat} if any node $\sigma\in T$ which is not a leaf has at least $n$ many immediate successors.
We use the following easy combinatorial fact, which is a slight modification of Cenzer-Hinman \cite[Proposition 2.9]{CH08} (see also Lemma \ref{lem:fat-tree} below).

\begin{fact}\label{fact:Cenzer-Hinman}
For $\ell,m\in\N$, let $T$ be an $(m\cdot\ell +1)$-fat finite tree, and $L_T$ be the set of all leaves of $T$.
Assume that every leaf of $T_\eta$ has the same length.
Then, for any function $f\colon L_T\to\ell$ there exists an $(m+1)$-fat tree $S\subseteq T$ such that $f$ is constant on the leaves of $S$.
\end{fact}

\begin{prop}\label{prop:llpo-error}
${\tt LLPO}_{1/\ell}\not\leqLT {\tt Error}_{1/\ell+2}$.
\end{prop}

\begin{proof}
The proof is by a typical recursion trick; see \cite{Kih20}.
Suppose for the sake of contradiction that there exists a winning \Art-\Ni strategy $(\tau\mid\eta)$ witnessing ${\tt LLPO}_{1/\ell}\leqLT {\tt Error}_{1/\ell+2}$.
Except for the first move $e$, \Mer's move is always a number $j<\ell+2$, which yields the tree $(\ell+2)^{<\N }$ of all possible moves by \Mer.
Fix $e$, and then \Art's strategy $\tau$ yields a partial computable function $\Phi_\tau\pcolon(\ell+2)^{<\N }\to\N $, where $\Phi_\tau(\sigma)\downarrow=u$ if and only if, after reading \Mer's moves $\sigma$, \Art's strategy $\tau$ declares termination with $u$.
Moreover, as \Ni makes a secret input $A\subseteq(\ell+2)$ for ${\tt Error}_{1/\ell+2}$ at each round, \Nim's strategy $\eta$ restricts \Mer's possible moves to an $(\ell+1)$-fat finite subtree $T_\eta$ of $(\ell+2)^{<\N }$, where after \Ar declares termination, \Ni makes no further moves; hence if $\sigma\in T_\eta$ and $\Phi_\tau(\sigma)\downarrow$ then $\sigma$ has to be a leaf of $T_\eta$.
On the other hand, if \Ar does not declare termination, then \Ni makes the next move, so $\Phi_\tau$ restricted to the leaves $L_\eta$ of $T_\eta$ yields a total function from $L_\eta$ to $\ell$.
One can assume that every leaf of $T_\eta$ has the same length:
Otherwise, let $t$ be the length of a longest node of $T_\eta$, and for each leaf $\rho$ of $T_\eta$ of length $s<t$, place a full $(\ell+1)$-branching tree of height $t-s$ on the leaf $\rho$.
Then, for a leaf $\rho$ of the resulting tree $T^\ast_\eta$, define $\Phi^\ast_\tau(\rho)$ as the value $\Phi_\tau(\rho^\ast)$ for the unique initial segment $\rho^\ast$ of $\rho$ such that $\rho^\ast\in L_\eta$.
Then replace $T_\eta$ with $T_\eta^\ast$ if necessary.

Since $\Phi_\tau$ is $\ell$-valued, and $T_\eta$ is $(\ell+1)$-fat, by Fact \ref{fact:Cenzer-Hinman}, there exists a $2$-fat subtree $S$ of $T_\eta$ of the same height such that $\Phi_\tau$ is constant on the leaves of $S$.
Recall that $\Phi_\tau$ depends on $e$, so there exists a computable function $d$ such that $\Phi_\tau=\varphi_{d(e)}$.
Now, we construct an algorithm $r(e)$ as follows:
By brute-force, \Me searches for a $2$-fat finite subtree $S$ of $(\ell+2)^{<\N }$ such that $\varphi_{d(e)}$ is total and constant on the leaves of $S$.
Let $j<\ell$ be the unique value of $\varphi_{d(e)}$ on the leaves of $S$.
Then, we declare that $\varphi_{r(e)}(j)$ halts, and $\varphi_{r(e)}(k)$ never halts for $k\not=j$.

Since \Nim's move reduces the number of possible moves for \Me by at most one at each round, and $S$ is $2$-fat, $S$ contains \Mer's correct moves in a play following $(\tau\mid\eta)$ as a path $\rho\in S$.
Since $(\tau\mid\eta)$ is winning, $\varphi_{d(e)}(\rho)=\Phi_\tau(\rho)=j\in{\tt LLPO}_{1/\ell}(e)$ since $e$ is \Mer's first move.
This means that $\varphi_e(j)\uparrow$.
However, by the Kleene recursion theorem (see e.g.~\cite[Theorem 4.1.1]{Cooper} or \cite[Theorem II.2.10]{OdiBook}), there exists $e$ such that $\varphi_e=\varphi_{r(e)}$, and by our definition, we have $\varphi_{r(e)}\downarrow$.
This is a contradiction.
\end{proof}

This also shows that ${\tt LLPO}^{\emptyset'}_{1/\ell}\not\leqLT {\tt Error}_{1/\ell+1}$.
%One can also show the following by modifying the argument in Cenzer-Hinman \cite[Theorem 2.10]{CH08}:

\begin{definition}\label{def:mul-plus-mono}
For a partial multifunction $f$ and a basic bilayer function $g$, we define a bilayer function $(f\mid g)$ as follows:
\[(f\mid g)(n\mid c)=f(n)\times g(\ast\mid c).\]
\end{definition}

A similar argument as above also shows the following:

\begin{prop}
${\tt Error}_{1/\ell}\not\leqLT (f\mid{\tt Error}_{1/\ell+1})$ for any partial multifunction $f$.
\end{prop}

\begin{proof}
Note that any partial multifunction $f$ is one-query Turing reducible to a single-valued function.
This is because any choice function for $f$ refines $f$.
Thus, without loss of generality, one can assume that $f$ is a single-valued function.

Suppose for the sake of contradiction that there exists there exists a winning \Art-\Ni strategy $(\tau\mid\eta)$.
Except for the first move $c$, the second coordinate of \Mer's move is always a number $j<\ell+1$, which yields the tree $(\ell+1)^{<\N }$ of all possible moves.
Given \Art's moves, the first coordinates of the corresponding moves by \Me are computable in $f$.
Therefore, \Art's strategy $\tau$ and the corresponding responses by \Me yield a partial $f$-computable function $\Phi^f_\tau\pcolon(\ell+1)^{<\N }\to\N $, where $\Phi^f_\tau(\sigma)\downarrow=u$ if and only if, after reading \Mer's moves whose second coordinates are $\sigma$, \Art's strategy $\tau$ declares termination with $u$.
Moreover, \Nim's strategy $\eta$ restricts second coordinates of \Mer's possible moves to an finite subtree $T_\eta$ of $(\ell+1)^{<\N }$.
As in the proof of Proposition \ref{prop:llpo-error}, one can assume that every leaf of $T_\eta$ has the same length $t$.
Then, $\Phi_\tau^f$ restricted to $T_\eta\cap (\ell+1)^t$ yields a total function from $T_\eta\cap (\ell+1)^t$ to $\ell$.
By considering $\min\{\Phi^f_\tau(\sigma),\ell-1\}$, one can assume that $\Phi^f_\tau$ is a partial $\ell$-valued function even if we consider an input $\sigma\not\in T_\eta\cap (\ell+1)^t$.
By Fact \ref{fact:Cenzer-Hinman} applied to any totalization of $\Phi^f_\tau$, one can see that there exists a $2$-fat subtree $S$ of $(\ell+1)^t$ of the same height such that $\Phi^f_\tau$ takes at most one value (or undefined) on the leaves of $S$ (where $\Phi_\tau^f$ can be partial).

Let $j<\ell$ be the only possible value of $\Phi_\tau^f$ on the leaves of $S$.
Note that this value only depends on \Art's strategy $\tau$; hence it is independent of \Mer's first move $c$.
Since \Nim's move reduces the number of possible moves for \Me by at most one at each round, and $S$ is $2$-fat, $S$ contains the second coordinates of \Mer's correct moves in a play following $(\tau\mid\eta)$ as a path $\rho\in S$.
Since $(\tau\mid\eta)$ is winning, $\Phi_\tau^f(\rho)$ is defined, and $j$ is the only possible value.
Hence $j\in{\tt Error}_{1/\ell}(c)$ since $c$ is \Mer's first move.
However, this value $j$ is independent of $c$, which is a contradiction.
\end{proof}

One can apply combinatorial techniques developed in \cite{CH08} to prove the following:

\begin{theorem}\label{thm:error-separation}
${\tt Error}_{m/k}\eqLT {\tt Error}_{1/\ell}$, where $\ell=\lceil\frac{k}{m}\rceil$.
\end{theorem}

This solves all problems asked in Lee-van Oosten \cite[Open problems in pages 876--877]{LvO}.

\begin{proof}[Proof of Theorem \ref{thm:error-separation}]
As in Cenzer-Hinman \cite[Proposition 2.4]{CH08}, one can easily see that $\lceil\frac{k}{m}\rceil\leq\ell$ implies ${\tt Error}_{1/\ell}\leqoLT {\tt Error}_{m/k}$.
Now, let us think of ${\tt Error}_{m/k}$ as a problem of choosing a surviving block, where there are $k$ blocks and one may secretly destroy at most $m$ of them.
As a variant of this problem, consider ${\tt Error}_{m/k;n}$, where there are $k$ blocks as above, but $n$ of them are hard blocks.
One can secretly hit blocks $m$ times, and while a normal block will break in one hit, a hard block will only break if we hit it all $m$ times.
The information about which blocks are hard is given as a public input $\langle a_i\rangle_{i<n}$, and the information on which blocks to hit is given as a secret input $\langle c_j\rangle_{j<m}$, where $a_i,c_j<k$ for any $i<n$ and $j<m$.
Formally,
\[
a\not\in{\tt Error}_{m/k;n}(\langle a_i\rangle_{i<n}\mid\langle c_j\rangle_{j<m})
\iff
\begin{cases}
a\in\{c_j\}_{j<m}&\mbox{ if }a\not\in\{a_i\}_{i<n},\\
\{a\}=\{c_j\}_{j<m}&\mbox{ if }a\in\{a_i\}_{i<n}.
\end{cases}
\]

This notion is an analogue of a sequence of type $(m,n)$ in \cite[Theorem 2.6]{CH08}.
Now assume that there are $k$ blocks, of which $n$ are hard blocks.
Consider an operation of consolidating $m$ of the normal blocks into a single hard block.
Then we now have $k-m+1$ blocks, since we have consolidated $m$ blocks into one.
The number of normal blocks remaining is $\ell=k-n-m$, and the number of hard blocks is $n+1$.
One can secretly hit a block $m$ times, but assuming that a hard block will always break if one hits it all the times does not change the difficulty of the problem, so one can assume that the number of times we hit it is $m^\ast=\min\{\ell,m\}$.
This consolidating operation transforms an instance of ${\tt Error}_{m/k;n}$ into an instance of ${\tt Error}_{m^\ast/k-m+1;n+1}$.

The join $f\sqcup g$ of bilayer functions $f$ and $g$ is defined as follows: 
\[
\begin{cases}
(f\sqcup g)(\langle 0,n\rangle\mid c)=f(n\mid c),\\
(f\sqcup g)(\langle 1,n\rangle\mid c)=g(n\mid c).
\end{cases}
\]

The next claim corresponds to the formula (6) in Cenzer-Hinman \cite[Theorem 2.6]{CH08}.

\begin{claim}
${\tt Error}_{m/k;n}\leqLT {\tt Error}_{m^\ast/k-m+1;n+1}\sqcup {\tt Error}_{m-1/k;n}$.
\end{claim}

\begin{proof}
Assume that an input $(\langle a_i\rangle_{i<n}\mid\langle c_j\rangle_{j<m})$ for ${\tt Error}_{m/k;n}$ is given.
There are only a finite number of patterns of consolidating $m$ normal blocks, and the location information of the $n$ hard blocks is given as a public input.
Therefore, by brute-force, \Ar can try all patterns of consolidating $m$ normal blocks.
If the number of patterns is $s$, at the first $s$ rounds, \Ar and \Ni make queries to ${\tt Error}_{m^\ast/k-m+1;n+1}$.
For such a round, \Ar chooses $m$ normal blocks $\langle b_i\rangle_{i<m}$, i.e., $\{b_i\}_{i<m}\cap\{a_i\}_{i<n}=\emptyset$, and consolidate them into a single hard block $b$.
Then \Ni hits the new blocks according to the original secret input $\langle c_j\rangle_{j<m}$; that is, if $c_j\not\in\{b_i\}_{i<m}$ then $c_j'=c_j$, and if $c_j\in\{b_i\}_{i<m}$ then $c_j'=b$.
Then, \Nim's next move is given by $\langle c_j'\rangle_{j<m}$.

If \Mer's response $u$ is not $b$ at some round, then $u$ is a solution to the original problem, so \Ar declares termination with $u$.
If \Mer's response $u$ is the new consolidated hard block $b$ at each round, then the original secret input $\langle c_j\rangle_{j<m}$ does not hit $m$ normal blocks.
This is because if $\langle c_j\rangle_{j<m}$ hits $m$ normal blocks then $\Ar$ consolidates these $m$ blocks into a single hard block $b$ at some round, and so \Ni hits the new hard block $m$ times.
This means that \Ni breaks $b$, so $b$ is not acceptable as \Mer's response.
Hence, the original secret input $\langle c_j\rangle_{j<m}$ hits at most $m-1$ normal blocks, so $(\langle a_i\rangle_{i<n}\mid\langle c_j\rangle_{j<m})$ can also be thought of as an input of ${\tt Error}_{m-1/k;n}$.
Thus, at round $s+1$, \Ar use $\langle a_i\rangle_{i<n}$ and \Ni use $\langle c_j\rangle_{j<m}$ as a query to ${\tt Error}_{m-1/k;n}$, and then \Mer's response $u$ must be a solution to $(\langle a_i\rangle_{i<n}\mid\langle c_j\rangle_{j<m})$.
Then, \Ar declares termination with $u$.
This is a winning \Art-\Ni strategy witnessing ${\tt Error}_{m/k;n}\leqLT {\tt Error}_{m^\ast/k-m+1;n+1}\sqcup {\tt Error}_{m-1/k;n}$.
\end{proof}

The next claim corresponds to the formula (11) in Cenzer-Hinman \cite[Theorem 2.6]{CH08}.

\begin{claim}
If $q=\lceil\frac{k-n}{m}\rceil+n$ then ${\tt Error}_{m/k;n}\leqLT {\tt Error}_{1/q}$.
\end{claim}

\begin{proof}
If $m=1$ then $q=k$ and there is no difference between normal blocks and hard blocks; hence ${\tt Error}_{1/k;n}\leqLT {\tt Error}_{1/k}$.
We prove the claim by induction on $m$ and $k$.
By the induction hypothesis, if we put $q_0=\lceil\frac{(k-m+1)-(n+1)}{m^\ast}\rceil+n+1$ and $q_1=\lceil\frac{k-n}{m-1}\rceil+n$, then
\begin{gather*}
{\tt Error}_{m^\ast/k-m+1;n+1}\leqLT {\tt Error}_{1/q_0}\\
{\tt Error}_{m-1/k;n}\leqLT {\tt Error}_{1/q_1}
\end{gather*}

We clearly have $q_1\geq q$, and moreover
\[
q_0=\left\lceil\frac{(k-m+1)-(n+1)}{m^\ast}\right\rceil+n+1\geq \left\lceil\frac{(k-m-n)}{m}\right\rceil+n+1=\left\lceil\frac{(k-n)}{m}\right\rceil+n=q.
\]

Thus, $q_0,q_1\geq q$ and this implies that ${\tt Error}_{1/q_0},{\tt Error}_{1/q_1}\leqoLT {\tt Error}_{1/q}$.
Hence, by the previous claim,
\begin{multline*}
{\tt Error}_{m/k;n}\leqLT {\tt Error}_{m^\ast/k-m+1;n+1}\sqcup {\tt Error}_{m-1/k;n}\\
\leqLT {\tt Error}_{1/q_0}\sqcup{\tt Error}_{1/q_1}
\leqLT {\tt Error}_{1/q}\sqcup{\tt Error}_{1/q}\eqLT{\tt Error}_{1/q}.
\end{multline*}

This verifies the claim.
\end{proof}

In particular, if we put $n=0$, then we have ${\tt Error}_{m/k}\leqLT {\tt Error}_{1/\ell}$ as $\ell=\lceil\frac{k}{m}\rceil$.
This concludes the proof of Theorem \ref{thm:error-separation}.
\end{proof}

\subsection{Non-existence of a minimal topology}

We next consider the basic bilayer function ${\tt Error}_{m/\N }$ for $m\in\N$ defined by
\begin{align*}
{\rm dom}({\tt Error}_{m/\N })&=\{(\ast\mid A):A\subseteq\N\;\land\;|A|=m\},\\
{\tt Error}_{m/\N }(\ast\mid A)&=\N\setminus A
\end{align*}

In Lee-van Oosten \cite[Propositions 5.1 and 5.2]{LvO}, it is shown that ${\tt Error}_{m/\N }\loLT{\tt Error}_{m+1/\N }$, but ${\tt Error}_{1/\N }\eqLT{\tt Error}_{m/\N }$.
Interestingly, as shown in \cite[Proposition 5.5]{LvO}, ${\tt Error}_{1/\N }$ is the $\leqLT $-least basic bilayer function which is strictly $\geqLT $-above ${\tt Id}$.
However, we include bilayer functions, ${\tt Error}_{1/\N }$ is not the $\leqLT $-least one.
For instance, the basic bilayer function ${\tt Error}_{1/\N }$ is closely related to the notion called {\em all-or-counique choice} \cite{pauly-handbook}:
\[
{\rm dom}({\tt ACC})=\N,\qquad
{\tt ACC}(e)=
\begin{cases}
\N\setminus\{\varphi_e(e)\}&\mbox{ if $\varphi_e(e)\downarrow$}\\
\N&\mbox{ if $\varphi_e(e)\uparrow$}
\end{cases}
\]

One may relativize ${\tt ACC}$ by replacing $\varphi_e$ with $\varphi_e^\alpha$ for a given oracle $\alpha\in 2^\N$, and then the resulting function is denoted by ${\tt ACC}^\alpha$.
Again, we think of a partial multifunction as a bilayer function as in Example \ref{exa:pmv-is-bilayer}.
The following is obvious:

\begin{obs}\label{obs:acc-error}
${\tt Id}\lLT{\tt ACC}^\alpha\lLT{\tt Error}_{1/\N }$ for any oracle $\alpha$.
\end{obs}

Lee \cite[Open Problem 3.5.18]{Lee} asked if there is the least topology strictly above ${\tt Id}$.
By Corollary \ref{cor:Turing-vs-LT}, it is the same as asking if there is the $\leqLT $-least bilayer function which is strictly $\geqLT $-above ${\tt Id}$.
If such a function exists, then it must be a non-basic bilayer function.
To solve this problem, given a partial function $g\pcolon\N\to\N$, we consider the following multifunction ${\tt Avoid}_g\colon\N\tto\N$:
\[{\tt Avoid}_g(n)=
\begin{cases}
\N\setminus\{g(n)\}&\mbox{ if $n\in{\rm dom}(g)$}\\
\N&\mbox{ otherwise.}
\end{cases}
\]

We first show that there exists no $\leq_T$-least partial multifunction which is strictly $\geq_T$-above ${\tt Id}$.

\begin{lemma}\label{lem:cone-avoiding}
Let ${\tt P}\pcolon\N\tto\N$ be a partial multifunction such that ${\tt P}>_T{\tt Id}$.
Then, there exists a total function $g\colon\N\to\N$ such that
\[{\tt Id}<_T{\tt Avoid}_g\quad\mbox{ and }\quad {\tt P}\not\leq_T{\tt Avoid}_g.\]
\end{lemma}

\begin{proof}
Let $\alpha_{\tt P}\in\N^\N$ be an oracle coding the full information of ${\tt P}$.
For instance,
\[
\alpha_{\tt P}(\langle n,m\rangle)=
\begin{cases}
0&\mbox{ if }n\in{\rm dom}({\tt P})\;\land\;m\not\in{\tt P}(n)\\
1&\mbox{ if }n\in{\rm dom}({\tt P})\;\land\;m\in{\tt P}(n)\\
2&\mbox{ if }n\not\in{\rm dom}({\tt P})\\
\end{cases}
\]

Let $g\colon\N\to 2$ be an $\alpha_{\tt P}$-generic real (that is, $g$ is contained in any dense $\alpha_{\tt P}$-computable open set in Cantor space $2^\N$, and such a $g$ exists by the Baire category theorem).
By genericity, for any computable function $f$, there are infinitely many $n$ such that $f(n)=g(n)$; hence ${\tt Id}<_T{\tt Avoid}_g$.
Suppose for the sake of contradiction that \Ar has a winning strategy $\Psi$ for ${\tt P}\leq_T{\tt Avoid}_g$.
Then we show the following claim:

\begin{claim}
For any $s\in\N$, there exist $n\in\N$ and $\sigma\in\N^{<\N}$ such that
\[\min\sigma>s\;\land\;\Psi(n,\sigma)\downarrow=\lrangle{1,y}\;\land\;{\tt P}(n)\downarrow\;\land\;y\not\in{\tt P}(n),\]
where recall that by $\Psi(n,\sigma)=\langle 1,y\rangle$ we mean that, after reading \Mer's moves $(n,\sigma)$, \Ar declares termination with $y$.
\end{claim}

\begin{proof}
Otherwise, there is a number $s$ refuting the claim.
Let $\Psi_1(n,\sigma)$ denote the second coordinate of $\Psi(n,\sigma)$.
Let us consider a sequence $a_0,a_1,\dots$ such that $a_{\ell+1}\in{\tt Avoid}_g(\Psi_1(n,a_0,a_1,\dots,a_\ell))$ for any $\ell$, i.e., $\langle a_i\rangle$ is \Mer's moves obeying the rule.
It is clear that there always exists such an $a_\ell>s$ since ${\tt Avoid}_g$ only reduces the number of possible values by one.
Since $\Psi$ is winning, after reading a finite initial segment of the sequence $\langle a_i\rangle$, \Ar declares termination; that is, $\Psi(n,a_0,a_1,\dots,a_\ell)=\lrangle{1,y}$ for some $\ell$ and $y$.

In particular, there exists $\sigma$ with $\min\sigma>s$ such that $\Psi(n,\sigma)\downarrow=\lrangle{1,y}$ for some $y$.
Given $n$, without knowing the information about $g$, one can effectively find such a $\sigma$ by brute-force.
Then define $f(n)$ as $\Psi_1(n,\sigma)$.
Clearly, $f$ is a computable function.
However, since the claim is supposed to fail, we have either ${\tt P}(n)\uparrow$ or $f(n)=y\in{\tt P}(n)$.
Hence, the computable function $f$ refines ${\tt P}$, which means ${\tt P}\leq_T{\tt Id}$.
However this contradicts our assumption on ${\tt P}$.
\end{proof}

Given $s$, one can find an $n$ and a $\sigma$ in the above claim in an $\alpha_{\tt P}$-computable manner since $\alpha_{\tt P}$ contains the full information of ${\tt P}$.
For a given $t$, put $s(t)=\max\{g(n):n<t\}$, and for this $s=s(t)$, we write $n_t$ and $\sigma_t=(a^t_i)_{i<\ell}$ for $n$ and $\sigma$ in the claim.
By the above claim, after reading \Mer's play $(n_t,\sigma_t)$, \Art's strategy $\Psi$ declares termination, but fails to compute a solution of ${\tt P}$.
However, since $\Psi$ is a winning strategy for ${\tt P}\leq_T{\tt Avoid}_g$, \Me must have violated the rule at some round.
In other words, there exists $j<\ell$ such that $a_{j+1}^t\not\in{\tt Avoid}_g(\Psi_1(n,a_0^t,a_1^t,\dots,a_j^t))$.
Put $m_j^t=\Psi_1(n,a_0^t,a_1^t,\dots,a_j^t)$, and we now have $a_{j+1}^t=g(m_j^t)$ for some $j<\ell$.
Note that, since $a_j^t>s$ for any $j$, if $m_j^t<t$ then $g(m_j^t)\not=a_{j+1}^t$ by our choice of $s=s(t)$.
Now consider the finite set $E_t$ defined by
\[E_t=\{\langle m_j^t,a_{j+1}^t\rangle:j<\ell\;\land\;m_j^t\geq t\}.\]

Given $t$, one can find the canonical code of $E_t$ is an $\alpha_{\tt P}$-computable manner.
Note that there exists $\langle m,a\rangle\in E_t$ such that $g(m)=a$.
However, note that
\[D=\{\tau\in 2^{<\N }:(\exists t\in\N)(\forall\langle m,a\rangle\in E_t)\;\tau(m)\not=a\}\]
yields a dense $\alpha_{\tt P}$-computable open set in Cantor space $2^\N$.
This is because, given a binary string $\tau$, choose $t>|\tau|$, and extend $\tau$ to $\tau^\ast$ so that $\tau^\ast(m)\not=a$ for any $\langle m,a\rangle\in E_t$.
This is doable since $\langle m,a\rangle\in E_t$ obviously implies $m\geq t$.
In this way, any $\tau$ extends to $\tau^\ast\in D$ and thus $D$ is dense.
Since $g$ is $\alpha_{\tt P}$-generic, we have $g\in D$; that is, there exists $t$ such that $g(m)\not=a$ for any $\langle m,a\rangle\in E_t$.
However, by the property of $E_t$, for any $t$, we must have a pair $\langle m^t_j,a^t_{j+1}\rangle\in E_t$ such that $g(m_j^t)=a_{j+1}^t$, a contradiction.
\end{proof}

\begin{theorem}
There exists no $\leqLT $-minimal bilayer function which is strictly $\geqLT $-above ${\tt Id}$; that is, for any bilayer function ${\tt P}\gLT{\tt Id}$ there exists a bilayer function ${\tt Q}$ such that ${\tt Id}\lLT{\tt Q}\lLT{\tt P}$.
Hence, there exists no minimal Lawvere-Tierney topology which is strictly above ${\tt Id}$.
\end{theorem}

\begin{proof}
Let ${\tt P}$ be a bilayer function such that ${\tt P}\gLT{\tt Id}$.
Put ${\tt P}_n(\ast\mid c)={\tt P}(n\mid c)$, and then ${\tt P}_n$ is a basic bilayer function.
If ${\tt P}_n\not\leqLT {\rm id}$, then by minimality of ${\tt Error}_{1/\N }$ (\cite[Proposition 5.5]{LvO}), we have ${\tt Error}_{1/\N }\leqLT {\tt P}_n\leqLT {\tt P}$.
By Observation \ref{obs:acc-error}, ${\tt P}$ cannot be minimal.
Thus, one can assume that ${\tt P}_n\leqLT {\tt Id}$ for any $n\in{\rm dom}({\tt P})$.
This implies that, since ${\tt P}_n$ is a basic bilayer function, if $n\in{\rm dom}({\tt P})$, then $\tilde{\tt P}(n):=\bigcap_c{\tt P}_n(\ast\mid c)$ is nonempty.
This is because, at some round in the reduction game for ${\tt P}_n\leqLT {\tt Id}$, \Art's winning strategy declares termination with some $u\in{\tt P}_n(\ast\mid c)$, which is independent of $c$, as $c$ is invisible to \Art.
Then, since $\tilde{\tt P}$ is a partial multifunction, by Lemma \ref{lem:cone-avoiding}, there exists a total function $g$ such that ${\tt Id}<_T{\tt Avoid}_g$ and $\tilde{\tt P}\not\leq_T{\tt Avoid}_g$.

We claim that ${\tt P}\not\leqLT{\tt Avoid}_g$ also holds.
Otherwise, there exists \Art's winning strategy $\tau$ witnessing ${\tt P}\leqLT {\tt Avoid}_g$, where note that, since ${\tt Avoid}_g$ is a partial multifunction, \Ni does not intervene in the game.
Let $(n\mid c)$ be \Mer's first move.
Since $\tau$ is winning, \Ar declares termination with some $u\in{\tt P}(n\mid c)$.
However, since $c$ is invisible to \Art, the last value $u$ only depends on $n$.
This implies that $u\in\tilde{\tt P}(n)=\bigcap_{c}{\tt P}(n\mid c)$.
Namely, the strategy $\tau$ also witnesses $\tilde{\tt P}\leq_T{\tt Avoid}_g$.
However, this contradicts our choice of $g$.

Moreover, since the Lawvere-Tierney topologies form a lattice, the Turing degrees of bilayer functions also form a lattice by Corollary \ref{cor:Turing-vs-LT}.
For an explicit description of the infimum operation, we define the meet ${\tt A}\sqcap{\tt B}$ of bilayer functions ${\tt A}$ and ${\tt B}$ as follows:
\[({\tt A}\sqcap{\tt B})(m,n\mid c,d)=\big(\{0\}\times{\tt A}(m\mid c)\big)\cup\big(\{1\}\times{\tt B}(n\mid d)\big).\]

Note that this is a bilayer analogue of the meet in the Weihrauch lattice \cite{pauly-handbook}.
We claim that the meet ${\tt Q}:={\tt P}\sqcap{\tt Avoid}_g$ of ${\tt P}$ and ${\tt Avoid}_g$ strictly lies between ${\tt Id}$ and ${\tt P}$.
As ${\tt P}\not\leqLT{\tt Avoid}_g$, we have ${\tt Q}\lLT{\tt P}$.
If ${\tt Id}\not\lLT{\tt Q}$, then ${\tt Q}\leqLT{\tt Id}$, so ${\tt Q}$ is computable.
This means that there exists a computable function $p$ such that $p(n,m)\in{\tt Q}(n,m\mid c,\ast)$ for $n,m,c$.
First consider the case that for any $n\in{\rm dom}({\tt P})$ there exists $m\in\N$ such that $p(n,m)$ is of the form $\langle 0,k\rangle$.
For such an $m$, since ${\tt Avoid}_g$ is total, we also have $\langle n,m\rangle\in{\rm dom}({\tt Q})$, and thus $p(n,m)\in{\tt Q}(n,m\mid c,\ast)$, so $k\in{\tt P}(n\mid c)$ for any $c$.
In this case, given $n\in{\rm dom}({\tt P})$, by brute-force, we effectively search for $m\in\N$ such that the first coordinate of $p(n,m)$ is $0$, then return its second coordinate, which must be a solution to ${\tt P}(n\mid c)$ as seen above.
Hence, this procedure witnesses that ${\tt P}$ is computable, which contradicts the assumption ${\tt P}\gLT{\tt Id}$.
Thus, there must exist $n\in{\rm dom}({\tt P})$ such that $p(n,m)$ is of the form $\langle 1,k\rangle$, for any $m\in\N$.
As in the above argument, we have $p(n,m)\in{\tt Q}(n,m\mid c,\ast)$, so $k\in{\tt Avoid}_g(m)$.
Then, the algorithm which, given $m\in\N$, returns the second coordinate of $p(n,m)$ witnesses that ${\tt Avoid}_g$ is computable.
However, this contradicts the property ${\tt Avoid}_g\gLT{\tt Id}$.
Consequently, ${\tt Id}\lLT{\tt Q}\lLT{\tt P}$.
This verifies the first assertion.
Then, the second assertion follows from Corollary \ref{cor:Turing-vs-LT}.
\end{proof}

This solves Lee's problem \cite[Open Problem 3.5.18]{Lee}. 

\section{Other topologies}\label{sec:other-topologies}

\subsection{Probabilistic computation}\label{sec:prob-comp-er}

As in Section \ref{sec:imperfect-information-game}, we consider bilayer functions expressing certain kinds of probabilistic computation.
However, unlike Section \ref{sec:imperfect-information-game}, for the sake of brevity in discussion, we require that a parameter $A$ be compact.
By inner regularity, every $\mu$-measurable set in $2^\N$ is approximated from the inside by a compact set, so adding this assumption does not make much difference.
Here, recall that $\mu$ is the uniform probability measure on $2^\N$.
Then we consider the following bilayer function:
\begin{gather*}
\begin{aligned}
{\tt ProbError}_\ep(\langle e,n\rangle\mid A)\downarrow\iff
&\mbox{$A\subseteq 2^\N$ is compact}\\
&\land\;\mu(A)\geq 1-\ep\;\land\;(\forall \alpha\in A)\;\varphi_e^\alpha(n)\downarrow.
\end{aligned}
\\
{\tt ProbError}_\ep(\langle e,n\rangle\mid A)=\{\varphi_e^\alpha(n):\alpha\in A\}.
\end{gather*}

One would say that the subtopos obtained from the Lawvere-Tierney topology corresponding to ${\tt ProbError}_\ep$ is the ``{\em world of probabilistically computable mathematics with error probability $\ep$}.''
Surprisingly, we show that ${\tt ProbError}$ induces exactly the same topology as Lee-van Oosten's function ${\tt Error}$.

\begin{prop}\label{prop:prob-error-and-error}
For any $p,q\in\N$ with $p\leq q$, ${\tt ProbError}_{p/q}\eqLT{\tt Error}_{p/q}$.
\end{prop}

\begin{proof}
Obviously, we have ${\tt Error}_{p/q}\leqLT{\tt ProbError}_{p/q}$.
We show the other direction.
Let $(\langle e,n\rangle\mid A)$ be an input for ${\tt ProbError}_{p/q}$ and assume that ${\tt ProbError}_{p/q}(\langle e,n\rangle\mid A)\downarrow$.
This means that $\varphi^\alpha_e(n)\downarrow$ for any $\alpha\in A$.
Since $A$ is compact, there exists $t$ such that the stage $t$ approximation $\varphi^{\alpha\upto t}_{e,t}(n)$ halts for any $\alpha\in A$.
Note that $\langle e,n\rangle\mapsto t$ is not necessarily computable, and thus \Ar does not have access to this information, but \Ni does have access to it.

The public input $\langle e,n\rangle$ can be thought of as a partial function $\alpha\mapsto\varphi_e^\alpha(n)$.
This induces the push-forward measure on $\N$ defined by $\nu(j)=\mu(\{\alpha\in 2^\N:\varphi_e^\alpha(n)\downarrow=j\})$.
We consider its finitary approximation; that is, put $E_j^s=\{\alpha\in 2^\N:\varphi_{e,s}^{\alpha\upto s}(n)\downarrow=j\}$, and define $\nu_s(j)=\mu(E_j^s)$.
For each stage $s$, the value $\nu_s(j)$ is rational, and moreover $\nu_s(j)>0$ happens for at most finitely many $j\in\N$.
Let $\langle j_i\rangle_{i<\ell(s)}$ be a list of all such $j$'s.
Clearly, $s\mapsto\langle\nu_s(j_i)\rangle_{i<\ell(s)}$ is computable.
Since we have only finitely many rationals $\langle\nu_s(j_i)\rangle_{i<\ell(s)}$, we can assume that all values have the same denominator and are of the form $\nu_s(j_i)=a_{s,i}/qr_s$.
Note that we clearly have $\sum_{i<\ell(s)}a_{s,i}\leq qr_s$.
Fix a pairwise disjoint sequence $(J^s_i)_{i<\ell(s)}$ such that $J^s_i\subseteq qr_s$ and $|J^s_i|=a_{s,i}$.

Since ${\tt Error}_{pr_s/qr_s}\leqLT {\tt Error}_{p/q}$ by Theorem \ref{thm:error-separation}, there is \Art's winning strategy $\tau_s$ witnessing this fact.
As $s\mapsto r_s$ is computable, one can easily see that the proof of Theorem \ref{thm:error-separation} ensures that $s\mapsto \tau_s$ is also computable.
Hence, instead of taking \Nim's moves as secret inputs to ${\tt Error}_{p/q}$, we can take them directly as secret inputs to ${\tt Error}_{pr_s/qr_s}$ through the above reduction implicitly.
Here, note that such a conversion takes multiple rounds through \Ar and \Nim's moves, but we do not count this number of rounds, and without mentioning it, all conversions are assumed to be done automatically.

Thus, at the $(s+1)$-st round, after seeing \Nim's previous move, we assume that \Me plays a move $j$ less than $qr_s$.
Then \Ar reacts to this.
If \Mer's move $j$ is contained in $J^s_i$ for some $i<\ell(s)$, then \Ar declares termination with $j_i$.
Otherwise, \Ar declares that the game proceeds to the next round, and urges \Ni to make the next move.
Now we describe \Nim's move $B_s$ at the $s$-th round.
\Ni reads the secret input $A$ and then define $B_s$ as follows:
For any $m\in J_i^s$, we define
\[m\in B_s\iff A\cap E^s_{j_i}=\emptyset.\]

For $m\in qr_s\setminus \bigcup_{i<\ell(s)}J_i^s$, we define
\[m\in B_s\iff A\cap E^s_\uparrow=\emptyset,\]
where $E^s_\uparrow=\{\alpha\in 2^\N:\varphi_{e,s}^{\alpha\upto s}(n)\uparrow\}$.
One can see that if $t\geq s$ then $qr_s\setminus\bigcup_{i<\ell(s)}J_i^s$ is included in $B_s$, where $t$ is a number mentioned in the first paragraph of this proof.
%If $s\geq t$ we enumerate all $m\in qr_s\setminus \bigcup_{i<\ell(s)}J_i^s$ into $B_s$.
Note that, in order to enumerate $a_{s,i}$ many elements in $J_i^s$ into $B_s$, the measure of $A$ needs to be removed by $\nu_s(j_i)=a_{s,i}/qr_s$.
Similarly, in order to enumerate $qr_s-\sum_{i<\ell_s}a_{i,s}$ elements in $qr_s\setminus \bigcup_{i<\ell(s)}J_i^s$ into $B_s$, the measure of $A$ needs to be removed by
\[\mu(E^s_\uparrow)=1-\sum_{i<\ell(s)}\mu(E_{j_i}^s)=\frac{qr_s-\sum_{i<\ell_s}a_{i,s}}{qr_s}.\]

Since the measure of $A$ is greater than or equal to $1-p/q$, the measure removed from $A$ should be at most $pr_s/qr_s$.
Therefore, only at most $pr_s$ many elements can be enumerated into $B_s$, so $(\ast\mid B_s)$ belongs to the domain of ${\tt Error}_{pr_s/qr_s}$.
Hence, \Nim's move $B_s$ obeys the rule.

We claim that \Ar and \Ni win this game along the play described above.
Note that \Ar declares termination by the $(t+1)$-th round.
This is because the complement of $\bigcup_{i<\ell(s)}J_i^s$ is included in $B_s$, and thus \Me must play a move from $\bigcup_{i<\ell(s)}J_i^s$ at the $(t+1)$-th round.
In response to this move, \Art's strategy described above declares termination.
We now assume that \Ar declares termination with $j_i^s$ at the $(s+1)$-st round.
In order for this to happen, \Mer's previous move must belong to $J_i^s$.
In such a case, \Nim's previous move $B_s$ must satisfy $J_i^s\not\subseteq B_s$.
This means that $A\cap E_{j_i}^s\not=\emptyset$.
In particular, there exists $\alpha\in A$ such that $\varphi_e^\alpha(n)=j_i$.
Hence, $j_i\in{\tt ProbError}_{p/q}(\langle e,n\rangle\mid A)$.
\end{proof}

One can also consider a counterpart of ${\tt ProbError}_\ep$ in the context of partial multifunctions, which is known as {\em weak weak K\"onig's lemma} \cite{BGH15}.
Let $P_e$ be the $e$-th $\Pi^0_1$ subset of $2^\N$ (or the set of all infinite paths though the $e$-th primitive recursive subtree of $2^{<\N}$).
\begin{gather*}
{\tt WWKL}_\ep(\langle e,n,i\rangle)\downarrow\iff \mu(P_i)\geq 1-\ep\;\land\;(\forall \alpha\in P_i)\;\varphi_e^\alpha(n)\downarrow.\\
{\tt WWKL}_\ep(\langle e,n,i\rangle)=\{\varphi_e^\alpha(n):\alpha\in P_i\}.
\end{gather*}

As in Proposition \ref{prop:prob-error-and-error}, one can show that ${\tt WWKL}_{p/q}\eqLT{\tt LLPO}_{p/q}$.
However, of course, weak weak K\"onig's lemma is known to be much stronger than the lessor limited principle of omniscience.
Indeed, if we consider analogues of ${\tt WWKL}$ and ${\tt LLPO}$ in the Kleene-Vesley algebra (i.e., in the context of $\N^\N$-computation), then one can easily see that ${\tt WWKL}$ is strictly above ${\tt LLPO}$ with respect to Turing reducibility for $\N^\N$-computability (i.e., generalized Weihrauch reducibility).
Therefore, Proposition \ref{prop:prob-error-and-error} is a phenomenon specific to the effective topos.
If we consider another (relative) realizability topos, such as the Kleene-Vesley topos, the situation would be completely different.

\begin{figure}[t]
{\footnotesize
\includegraphics[width=9cm]{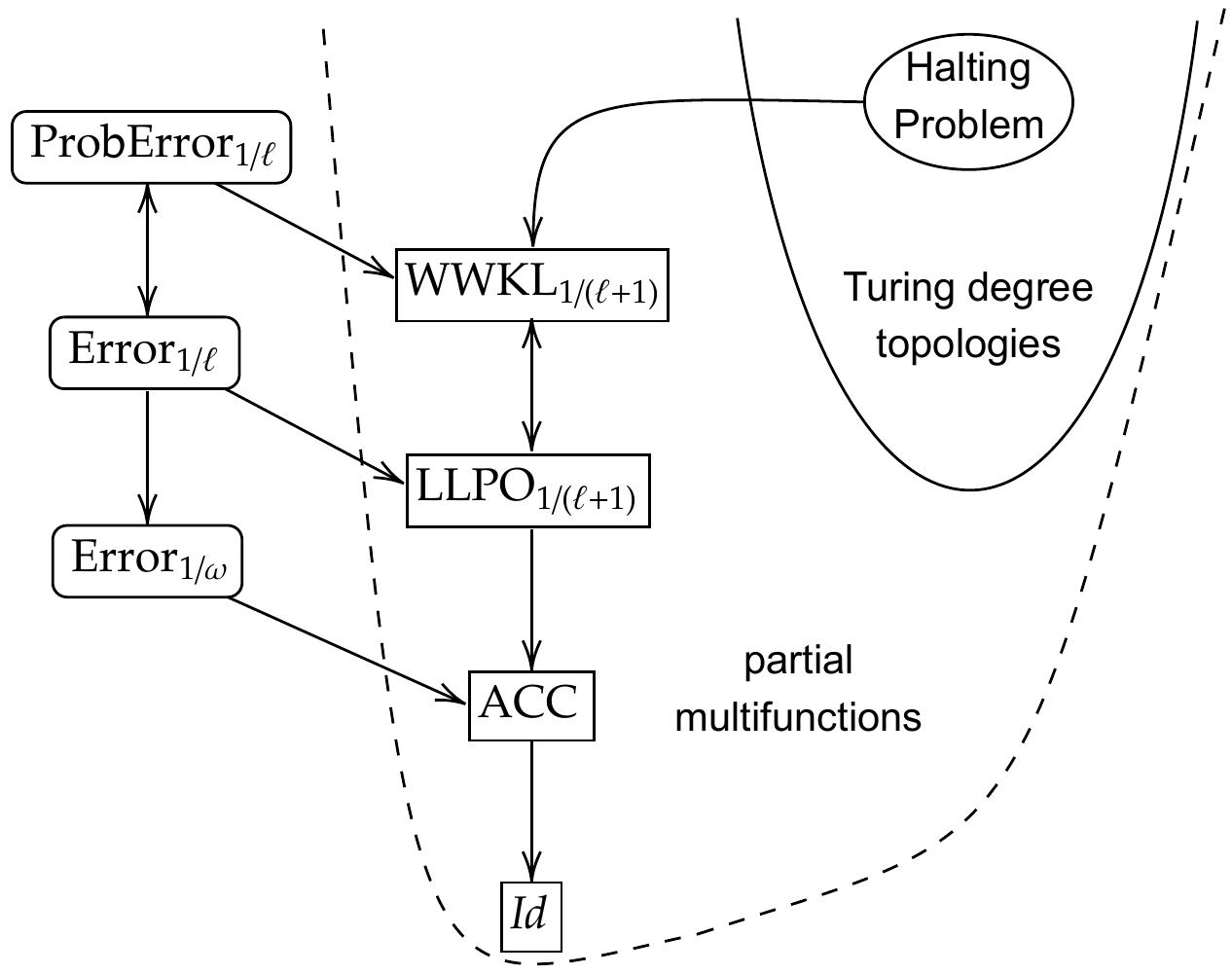}
\caption{\footnotesize Lower parts on Lawvere-Tierney topologies on the effective topos}\label{figure:topology}
}
\end{figure}

Figure \ref{figure:topology} summarizes some of basic implications about the $\leq_{\rm rea}$-ordering on the Lawvere-Tierney topologies on the effective topos, where $A\to B$ means $B\leqLT A$ (or $B^{\Game\to}\leq_{\rm rea}A^{\Game\to}$).
By our results, there are no more implications in Figure \ref{figure:topology}.

\subsection{Cofinite choice}

The basic bilayer functions we have dealt with so far was of no help at all for computing a single-valued function when treated as an oracle.
However, indeed, there is a known basic bilayer function which is rather powerful when considered as an oracle.
Pitts \cite[Example 5.8]{Pit85} introduced the following basic bilayer function ${\tt Cofinite}$:
\[{\rm dom}({\tt Cofinite})=\{\ast\}\times\N,\qquad{\tt Cofinite}(\ast\mid n)=\{m\in\N:m\geq n\}.\]

Note that one-query ${\tt Cofinite}$-relative computation is the one in which countably many computations are run in parallel, a finite number of which may be wrong.
Pitts \cite{Pit85} observed that the above function yields a Lawvere-Tierney topology $\mathcal{J}$ on the effective topos such that ${\rm id}<_{\rm rea}\mathcal{J}<_{\rm rea}\neg\neg$.
Here, note that $\mathcal{J}={\tt Cofinite}^{\Game\to}$ in our terminology.
Interestingly, van Oosten \cite[Theorem 2.2]{vO14} showed that, for a total function $f\colon\N\to\N$, the topology $\mathcal{J}$ forces $f$ to be decidable if and only if $f$ is hyperarithmetic.
In other words, for a total function $f\colon\N\to\N$, $f\leqLT {\tt Cofinite}$ if and only if $f$ is hyperarithmetic.
For the basics of hyperarithmetic sets, we refer the reader to Sacks \cite{sacks2} and Chong-Yu \cite{ChYuBook}.
It is also known that ${\tt Error}_{1/\ell}\not\leqLT {\tt Cofinite}$ for any $\ell\in\N$ (see \cite[Proposition 5.11]{LvO}).

Pitts' function ${\tt Cofinite}$ also has a partial multi-valued counterpart, which has been studied as {\em cofinite choice} \cite{BG11,AK19} or {\em bound} \cite{GPV} in the context of $\N^\N$-computation.
We introduce cofinite choice relative to an oracle $\alpha\in 2^\N$ as follows:
\begin{gather*}
{\tt C}^\alpha_{\tt cof}(e)\downarrow\iff |\{n\in\N:\varphi^\alpha_e(n)\downarrow\}|\mbox{ is finite},\\
{\tt C}^\alpha_{\tt cof}(e)=\{n\in\N:\varphi^\alpha_e(n)\uparrow\}.
\end{gather*}

It is obvious that ${\tt C}^\alpha_{\tt cof}\leqLT{\tt Cofinite}$ for any oracle $\alpha\in 2^\N$.
Moreover, we also have ${\tt C}^\alpha_{\tt cof}\leq_T\alpha'$, where $\alpha'$ is the Turing jump of $\alpha$.
To estimate the strength of cofinite choice, we consider the following equivalent definition of the hyperarithmetical hierarchy based on the effective Baire hierarchy.

\begin{definition}[Effective Baire hierarchy]\label{def:eff-Baire-hie}
For each computable ordinal $\xi$, we define a set $B_\xi$ of total functions on $\N$ as follows:
First, $B_0$ is the set of all total computable functions on $\N$, and a $B_0$-code of $f\in B_0$ is a program code computing $f$.
For $\xi>0$, $B_\xi$ is the set of functions $f\colon\N\to\N$ such that $f=\lim_{s\to\infty}f_s$ for some sequence $(f_n)_{n\in\N}$, where there exists an algorithm $\Phi_f$ which, given $n\in\N$, returns a $B_\zeta$-code of $f_n\in B_\zeta$ for some $\zeta<\xi$.
A $B_\xi$-code of $f\in B_\xi$ is a pair of a code of $\xi$ and a code of such $\Phi_f$.
\end{definition}

By the Shoenfield limit lemma (see e.g.~\cite[Lemma III.3.3]{SoareBook} or \cite[Proposition IV.1.19]{OdiBook}), $B_{n}$ corresponds to $\emptyset^{(n)}$-computability for $n\in\N$, and $B_{\xi}$ corresponds to $\emptyset^{(\xi+1)}$-computability for an infinite ordinal $\xi$, where $\emptyset^{(\xi)}$ is the $\xi$-th Turing jump of a computable function.
Obviously, Definition \ref{def:eff-Baire-hie} of $f\in B_\xi$ produces a computable well-founded tree $T_f$ whose leaves are labeled by computable functions.
Here, such a $T_f$ is {\em full-splitting}, that is, if $\sigma\in T_f$ is not a leaf, then $\sigma\fr n\in T_f$ for any $n\in\N$.

Conversely, let $T\subseteq\N^{<\N}$ be a computable full-splitting well-founded tree.
We inductively define ${\rm rank}_T(\sigma)$, the rank of $\sigma$, as follows:
The rank of a leaf of $T$ is $0$, and if $\sigma\in T$ is not a leaf, then the rank of $\sigma$ is $\sup\{{\rm rank}_T(\sigma\fr n)+1:n\in\N\}$.
Then, ${\rm rank}(T)$, the rank of $T$, is defined as the rank of its root.

Let $L_T$ be the set of leaves of $T$.
A {\em computable assignment} is a computable function $h\colon L_T\times\N\to\N$.
We inductively label each node $\sigma$ of $T$ with a function $h_\sigma\colon\N\to\N$ or an undefined symbol $\uparrow$ as follows:
If $\rho$ is a leaf of $T$, then put $h_\rho(n)=f(\rho,n)$.
If $\sigma\fr s$ is labeled by $\uparrow$ for some $s\in\N$, then $h_\sigma$ is also labeled by $\uparrow$.
If $\lim_{s\to\infty}h_{\sigma\fr s}(n)$ exists for all $n\in\N$, then $\sigma$ is labeled by the pointwise limit $h_\sigma=\lim_{s\to\infty}h_{\sigma\fr s}$.
Otherwise, $\sigma$ is labeled by $\uparrow$.
Then define $h_T$ as the label of the root of $T$ if it is defined.
Observe that if $h_\sigma$ is defined, then $h_\sigma\in B_{{\rm rank}_T(\sigma)}$.
Indeed, one can compute a $B_{{\rm rank}_T(\sigma)}$-code of $h_\sigma$.

\begin{definition}
We call a pair $(T,h)$ of a computable full-splitting well-founded tree $T$ and a computable assignment $h$ a {\em blueprint}.
\end{definition}

One can easily see that $f\in B_\xi$ if and only if there exists a blueprint $(T,h)$ such that $f=h_T$ and ${\rm rank}(T)=\xi$.
We are now ready to prove the following:

\begin{prop}
For any computable ordinal $\xi$, we have ${\tt C}_{\sf cof}^{\emptyset^{(\xi)}}\equiv_T\emptyset^{(\xi+1)}$.
\end{prop}

\begin{proof}
It suffices to show that $\emptyset^{(\xi+1)}\leq_T{\tt C}_{\sf cof}^{\emptyset^{(\xi)}}$.
As mentioned above, $\emptyset^{(\xi+1)}$ corresponds to $B_{1+\xi}$.
Therefore, it suffices to show that $f\leq_T{\tt C}_{\sf cof}^{\emptyset^{(\xi)}}$ for any $f\in B_{1+\xi}$.
Let $(T,h)$ be a blueprint defining $f$ as above.
Since $h_T=f$ is defined, $h_\sigma$ is also defined for any $\sigma\in T$ by definition.

We describe \Art's strategy for $f\leq_T{\tt C}_{\sf cof}^{\emptyset^{(\xi)}}$.
Assume that \Mer's first move is $n$, and the second and subsequent moves are $\sigma=\langle x_1,x_2,\dots,x_\ell\rangle$.
If $\sigma$ is a leaf of $T$, then \Ar declares termination with $h_\sigma(n)$.
Here, \Ar can use the information of $h_\sigma$ since $h_\sigma$ is computable.
Assume that $\sigma$ is not a leaf of $T$.
Since the rank of $T$ is $1+\xi$, the rank of $\sigma$ is at most $1+\xi$ and thus the rank of any immediate successor $\sigma\fr s$ of $\sigma$ is less than $1+\xi$.
By the property of a blueprint, given $s\in\N$, one can compute a $B_\zeta$-code of $h_{\sigma\fr s}\in B_\zeta$ for some $\zeta<1+\xi$.
In particular, $(s,n)\mapsto h_{\sigma\fr s}(n)$ is $\emptyset^{(\xi)}$-computable.

Then consider the following program $e$:
Given an input $s$, the computation searches for $t>s$ such that $h_{\sigma\fr t}(n)\not=h_{\sigma\fr s}(n)$.
If such a $t$ is found, the computation $\varphi_e^{\emptyset^{(\xi)}}(s)$ halts.
Otherwise, $\varphi_e^{\emptyset^{(\xi)}}(s)$ never halts.
Since $h_\sigma$ is defined, $\lim_{s\to\infty}h_{\sigma\fr s}(n)$ exists; that is, there exists $s$ such that $h_{\sigma\fr t}(n)=h_{\sigma\fr s}(n)$ for any $t>s$.
Hence, there are at most finitely many $s$ such that $\varphi_e^{\emptyset^{(\xi)}}(s)$ halts.
This means that $e\in{\rm dom}({\tt C}_{\tt cof}^{\emptyset^{(\xi)}})$.
Then, \Ar declares that the game is to continue, and uses $e$ as the next query, that is, $\langle 0,e\rangle$ is \Art's next move.

We claim that this is \Art's winning strategy.
To see this, if \Mer's second and subsequent moves are $\sigma=\langle x_1,x_2,\dots,x_\ell\rangle$, then we inductively show that $h_T(n)=h_\sigma(n)$.
We inductively assume that $h_T(n)=h_\sigma(n)$.
As the next move, if \Ar declares that the game is to continue, and uses $e$ as the next query, \Me responds to this with some $t_{\ell+1}\in{\tt C}_{\tt cof}^{\emptyset^{(\xi)}}(e)$.
This means that $h_{\sigma\fr t_{\ell+1}}(n)=\lim_{s\to\infty}h_{\sigma\fr s}(n)=h_\sigma(n)$.
Hence, by the induction hypothesis, we have $h_{\sigma\fr t_{\ell+1}}(n)=h_T(n)$.

If the history of \Mer's moves $\rho=\langle x_1,x_2,\dots,x_k\rangle$ reaches a leaf of $T$, then \Ar declares termination with $h_\rho(n)$, and by the above property, we obtain $h_\rho(n)=h_T(n)=f(n)$.
Hence, the procedure described above is shown to be \Art's winning strategy.
Consequently, we get $f\leq_T{\tt C}_{\tt cof}^{\emptyset^{(\xi)}}$.
\end{proof}

As ${\tt C}_{\tt cof}^\alpha\leqLT{\tt Cofinite}$ for any oracle $\alpha\in 2^\N$, this explains the reason why ${\tt Cofinite}$ is so powerful as an oracle.
Interestingly, by the result in van Oosten \cite[Theorem 2.2]{vO14} mentioned above, one can observe that even if relativized by a tremendously powerful oracle $\alpha$, ${\tt C}_{\tt cof}^\alpha$ never be able to compute a non-hyperarithmetic function; that is, for any non-hyperarithmetic $f$, we have $f\not\leq_T{\tt C}_{\tt cof}^\alpha$ no matter what an oracle $\alpha$ is.
Roughly speaking, this is because a computational process beyond hyperarithmetic is not a finite approximation process, but an approximation process along an ordinal, which prevents us from using ``time trick''; see e.g.~\cite{BGM21}.

\subsection{Asymptotic density}

As a candidate for another basic bilayer function, one that uses asymptotic density may come to mind.
For a set $A\subseteq\N$, the {\em lower asymptotic density} of $A$ is defined by 
\[\underline{d}(A)=\liminf_{n\to\infty}\frac{|A\cap n|}{n}.\]

Then, for any real $\ep\in[0,1]$, we define the basic bilayer function ${\tt DenError}_{\ep}$ as follows:
\begin{gather*}
{\rm dom}({\tt DenError}_\ep)=\{(\ast\mid A):A\subseteq\N\mbox{ and }\underline{d}(A)\geq 1-\ep\}\\
{\tt DenError}_\ep(\ast\mid A)=A.
\end{gather*}

Obviously, we have ${\tt Cofinite}\leqoLT {\tt DenError}_\ep$ for any real $\ep\in[0,1]$ since the asymptotic density of a cofinite set is $1$.
The major difference between ${\tt Cofinite}$ and ${\tt DenError}_\ep$ is the following property.

\begin{obs}\label{obs:err-vs-denerr}
${\tt Error}_{1/\ell}\leqoLT {\tt DenError}_{1/\ell}$ for any $\ell\in\N$.
\end{obs}

\begin{proof}
We define an outer reduction $K$ as follows:
For $n\in\N$ and $m<\ell$, put $K(n\ell+m)=m$.
Then, for an input $(\ast\mid\{j\})$ for ${\tt Error}_{1/\ell}$, a secret inner reduction $L$ is defined by $L(\{j\})=\{n\ell+m:n\in\N\mbox{ and }j\not=m<\ell\}$.
Note that the asymptotic density of $L(\{j\})$ is $1-1/\ell$.
Clearly, $y\in {\tt DenError}_{1/\ell}(L(\{j\}))=L(\{j\})$ implies $K(y)\not=j$; hence $K(y)\in{\tt Error}_{1/\ell}(\ast\mid\{j\})$.
\end{proof}

We say that a partial multifunction ${\tt P}\pcolon\N\tto\N$ is hyperarithmetic if there exists a partial $\Pi^1_1$ function $f\pcolon\N\to\N$ such that $f(n)\in{\tt P}(n)$ for any $n\in{\rm dom}({\tt P})$.

\begin{prop}\label{prop:denerror-hyp}
Let ${\tt P}$ be a partial multifunction whose codomain is $\ell\in\N$ with $\ell>0$.
For any $\ep<1/(\ell+1)$, if ${\tt P}\leqLT {\tt DenError}_\ep$, then ${\tt P}$ is hyperarithmetic.
\end{prop}

To prove this, we need to show an auxiliary lemma.
For a tree $T\subseteq\N^{<\N}$, let ${\rm succ}_T(t)$ be the set of all immediate successors of $t$ in $T$.
For a function $b\colon\N^{<\N}\to\N$, a tree $T$ is {\em $b$-fat} if, for any $t\in T$ which is not a leaf, $t$ has at least $b(t)$ immediate successors, i.e., $|{\rm succ}_T(t)|\geq b(t)$.
For a function $b\colon\N^{<\N}\to\N$ and $\ell\in\N$, consider the function $\ell\cdot b\colon\sigma\mapsto\ell\cdot b(\sigma)$.
The following is an analogue of Cenzer-Hinman \cite[Proposition 2.9]{CH08}.

\begin{lemma}\label{lem:fat-tree}
Let $b\colon\N^{<\N}\to\N$ be a function, $T$ be an $(\ell\cdot b)$-fat finite tree, and $L_T$ be the set of all leaves of $T$.
Then, for any function $f\colon L_T\to\ell$ there exists a $b$-fat tree $S\subseteq T$ such that $f$ is constant on the leaves of $S$.
\end{lemma}

\begin{proof}
Since $T$ is finite, one can assume that any $\sigma\in L_T$ has the same length.
We prove the assertion by induction on the height $k$ of $T$.
Let $f\colon L_T\to\ell$ be given.
If $k=1$, then $T$ has $\ell\cdot b(\ep)$ leaves, where $\ep$ is the empty strings.
As $f$ is $\ell$-valued, there are at least $b(\ep)$ leaves on which $f$ is constant.

Next, assume that $k>1$.
Then, for any $t\in T$ of length $k-1$, define $g(t)$ as the least $j<\ell$ such that $|\{n\in{\rm succ}_T(t):f(t\fr n)=j\})|\geq b(t)$.
Since $T$ is $(\ell\cdot b)$-fat, we have $|\{n:f(t\fr n)<\ell\}\}|\geq \ell\cdot b(t)$, so such a $j$ exists.
Note that $g$ is a function from $T\cap \N^{k-1}$ to $\ell$.
By the induction hypotheis, there exists a $b$-fat tree $S^-\subseteq T\cap \N^{<k}$ such that $g$ is constant on the leaves of $S^-$.
Let $j<\ell$ be the unique value of $g$ on the leaves $L_S^-$ of $S^-$.
Note that $L_S^-\subseteq T$.
Then, define
\[S=S^-\cup\{t\fr n:t\in L_S^-\;\land\;n\in{\rm succ}_T(t)\;\land\;f(t\fr n)=j\}\]

By definition, clearly $f$ is constant on the leaves of $S$.
Moreover, for any $t\in L_S^-$, by our definition of $g$ and $S^-$, we have 
\[|{\rm succ}_S(t)|=|\{n\in{\rm succ}_T(t):f(t\fr n)=j\}|\geq b(t).\]

Therefore, $S$ is $b$-fat.
This concludes the proof.
\end{proof}

\begin{proof}[Proof of Proposition \ref{prop:denerror-hyp}]
Let $(\tau\mid\eta)$ be a winning \Art-\Ni strategy witnessing ${\tt P}\leqLT {\tt DenError}_\ep$.
Except for the first move $n$, \Mer's move is always a number $j\in\N$, which yields the tree $\N^{<\N }$ of all possible moves by \Mer.
Fix $n$, and then \Art's strategy $\tau$ yields a partial computable function $\Phi^n_\tau\pcolon\N^{<\N }\to\N $, where $\Phi^n_\tau(\sigma)\downarrow=u$ if and only if, after reading \Mer's moves $\sigma$, \Art's strategy $\tau$ declares termination with $u$.
\Nim's strategy $\eta$ restricts \Mer's possible moves to a well-founded subtree $T^n_\eta$ of $\N^{<\N}$ such that, for each $\sigma\in T^n_\eta$, if $\sigma$ is a leaf then $\Phi^n_\tau(\sigma)$ is defined, and the lower asymptotic density of the set $A^n_\eta(\sigma):={\rm succ}_{T^n_\eta}(\sigma)$ is at least $1-\ep$ since \Ni obeys the rule as long as \Me obeys the rule, and this value is greater than $1-1/(\ell+1)$ since $\ep<1/(\ell+1)$.
By the definition of lower asymptotic density, there exists $b(\sigma)$ such that, for any $m\geq b(\sigma)$, we have $|A^n_\eta(\sigma)\cap m|/m>1-1/(\ell+1)$.
In particular, we have
\[|A^n_\eta(\sigma)\cap (\ell+1)\cdot b(\sigma)|>(\ell+1)\cdot b(\sigma)\cdot \left(1-\frac{1}{\ell+1}\right)=\ell\cdot b(\sigma).\]

Fix such $b\colon\N^{<\N}\to\N$, where if $\sigma$ is either a leaf of $T^n_\eta$ or $\sigma\not\in T^n_\eta$, then $b(\sigma)$ is arbitrary.
Then, consider the tree $T^{b}$ of $((\ell+1)\cdot b)$-bounded strings:
\[T^{b}=\{t\in\N^{<\N}:(\forall s<|t|)\;t(s)<(\ell+1)\cdot b(t\upto s)\}.\]

Since $T^b$ is finite branching and $T^n_\eta$ is well-founded, $T^{n,b}_\eta:=T^b\cap T^n_\eta$ is finite by K\"onig's lemma.
Moreover, $\Phi_\tau^n$ is total on the leaves of $T^{n,b}_\eta$.
Note that if $\sigma$ is not a leaf of $T^{n,b}_\eta$ then ${\rm succ}_{T^{n,b}_\eta}(\sigma)=A^n_\eta(\sigma)\cap (\ell+1)\cdot b(\sigma)$.
Hence, $T_\eta^{n,b}$ is $(\ell\cdot b)$-fat.
Since $\Phi_\tau^n$ is $\ell$-valued, by Lemma \ref{lem:fat-tree}, there exists a $b$-fat subtree $S$ of $T^{n,b}_\eta$ of the same height such that $\Phi^n_\tau$ is constant on the leaves of $S$.
Hereafter, we write $b$ as $b_n$ since such a $b$ satisfying the above density condition is depend on \Mer's first move $n$.

For a function $g\colon\N^{<\N}\to\N$,
%define $g_n(\sigma)=g(n\fr\sigma)$.
our algorithm $\Psi^{g}(n)$ searches for a $g$-fat finite tree $S\subseteq T^{g}$ such that $\Phi^n_\tau$ is constant on the leaves of $S$, and returns the unique value $j$ of $\Phi_\tau^n$ on the leaves of $S$ if such an $S$ exists.
Note that if $g(\sigma)$ is a correct witness for the above density condition for $A^n_\eta(\sigma)$, then such an $S$ always exists.
Now, we define $Q$ as follows
\[Q(n,j)\iff (\forall f\colon\N^{<\N}\to\N)(\exists g\geq f)\;[\Psi^g(n)\downarrow=j],\]
where by $f\leq g$ we mean that $f(\sigma)\leq g(\sigma)$ for any $\sigma\in\N^{<\N}$.
Note that $\Psi^g(n)$ is defined only if $\Psi^g(n)$ succeeds in finding $S$, which means that the algorithm $\Psi$ only reads $g$ up to the height of $S$.
%Since $\Psi$ is continuous, $\Psi$ only reads some initial segment of $q$ before $\Psi^g(n)$ is determined.
Thus, the above predicate $Q$ is $\Pi^1_1$.

We claim that $Q(n,j)$ implies $j\in{\tt P}(n)$.
%To see this, let $h(n\fr\sigma)$ be a sufficiently large number satisfying the above density condition, so that $|A^n_\eta(\sigma)\cap (\ell+1)\cdot m|>\ell\cdot m$ for any $m\geq h(n\fr\sigma)$.
To see this, put $f(\sigma)=b_n(\sigma)$.
If $g\geq f$ then we have $|A^n_\eta(\sigma)\cap(\ell+1)\cdot g(\sigma)|>\ell\cdot g(\sigma)$ by our choice of $b_n$.
If $\Psi^g(n)\downarrow=j$ then the algorithm $\Psi$ succeeds in finding a $g$-fat finite tree $S\subseteq T^{g}$ such that $\Phi^n_\tau(\rho)=j$ for any leaf $\rho$ of $S$.
Note that ${\rm succ}_S(\sigma)\subseteq(\ell+1)\cdot g(\sigma)$ by the definition of $T^g$, and $|{\rm succ}_S(\sigma)|\geq g(\sigma)$ since $S$ is $g$-fat.
This implies that ${\rm succ}_S(\sigma)\cap A_\eta^n(\sigma)$ is nonempty.
Hence, $S\cap T_\eta^n$ has a common leaf $\rho$.
Since $(\tau\mid\eta)$ is winning, and $\rho$ is a leaf of $T_\eta^n$, we must have $\Phi_\tau^n(\rho)\in{\tt P}(n)$.
By our choice of $S$, $\Phi_\tau^n$ is constant on the leaves of $S$, and thus $\Psi^g(n)$ must be equal to $\Phi_\tau^n(\rho)$.
This concludes $\Psi^g(n)=j\in{\tt P}(n)$.

We next claim that for any $n$ there exists $j<\ell$ such that $Q(n,j)$.
Otherwise, for any $j<\ell$ there exists $f_j$ such that either $\Psi^g(n)$ is undefined or $\Psi^g(n)\not=j$ for any $g\geq f_j$.
Then, put $h(\sigma)=\max\{b_n(\sigma),f_j(\sigma):j<\ell\}$.
As $h(\sigma)\geq b_n(\sigma)$, clearly, $h(\sigma)$ is a correct witness for the above density condition for $A^n_\eta(\sigma)$.
Then, by the argument using Lemma \ref{lem:fat-tree} described above, the algorithm $\Psi$ succeeds in finding $S$, so that $\Psi^h(n)\downarrow=j$ for some $j<\ell$.
However, as $h\geq f_j$, this contradicts our assumption on $f_j$.
Therefore, $Q$ determines a total relation.
Since $Q$ is $\Pi^1_1$, by $\Delta^1_1$-selection (see Moschovakis \cite[4B.5]{MosBook} or Sacks \cite[Theorem II.2.3]{sacks2}), there exists a hyperarithmetic function $p\colon\N\to\N$ such that $Q(n,p(n))$ holds.
By the first claim, this implies that $p(n)\in{\tt P}(n)$.
\end{proof}

In particular, ${\tt DenError}_\ep$ and ${\tt Cofinite}$ have partly the same properties in the following sense.

\begin{cor}
Assume $\ep<1/2$.
Then, for a function $f\colon\N\to\N$, $f\leqLT {\tt DenError}_\ep$ if and only if $f$ is hyperarithmetic.
\end{cor}

Proposition \ref{prop:denerror-hyp} also shows that ${\tt LLPO}_{1/\ell}^\alpha\not\leqLT {\tt DenError}_{1/(\ell+2)}$ for a sufficiently powerful oracle $\alpha$.
By Observation \ref{obs:llpo-from-error-pos}, this implies that ${\tt Error}_{1/\ell}\not\leqLT {\tt DenError}_{1/(\ell+2)}$.
Now it is natural to ask whether ${\tt Error}_{1/\ell}\leqLT {\tt DenError}_{1/(\ell+1)}$ or not.
One can answer this question by introducing the concept of hyperarithmetical reducibility for bilayer functions.
First consider the one-query version.

\begin{definition}
Let $f$ and $g$ be bilayer functions.
We say that {\em $f$ is a one-query hyperarithmetically LT-reducible to $g$} (written $f\leq_{hLT}^1g$) if there exist partial $\Pi^1_1$ functions $H$ and $K$ and a partial function $L$ such that for any $(n\mid c)$ and $m$,
\[m\in g(H(n)\mid L(n,c))\implies K(n,m)\in f(n\mid c).\]
\end{definition}

Let $\psi_e\pcolon\N\to\N$ be the $e$th patrial $\Pi^1_1$ function (given by the canonical enumeration of all $\Pi^1_1$ sets).
Then, consider the following partial multifunction:
\begin{align*}
{\rm dom}(\Pi^1_1\mbox{-}{\tt LLPO}_{m/k})&=\{e\in\N:|\{j<k:\psi_e(j)\downarrow\}|\leq m\},\\
\Pi^1_1\mbox{-}{\tt LLPO}_{m/k}(e)&=\{0,\dots,k-1\}\setminus \{j<k:\psi_e(j)\downarrow\}.
\end{align*}

It is well-known that $\Pi^1_1$ is higher analogue of computable enumerability, i.e., a set is $\Pi^1_1$ if and only if there exists a hyperarithmetical enumeration procedure along a computable ordinal; see e.g.~\cite{sacks2,MosBook}.
The following is a hyperarithmetical analogue of Proposition \ref{prop:llpo-vs-error-pos}:

\begin{prop}\label{prop:pillpo-vs-error-pos}
$\Pi^1_1$-${\tt LLPO}_{m/k}\leq_{hLT}^1{\tt Error}_{m/k+1}$.
\end{prop}

\begin{proof}
We define a secret inner reduction $L$ as follows:
For any $e\in{\rm dom}(\Pi^1_1\mbox{-}{\tt LLPO}_{m/k})$,
\[
L(e)=
\begin{cases}
\{j<k\colon\psi_e(j)\downarrow\}&\mbox{ if }|\{j<k\colon\psi_e(j)\downarrow\}|=m\\
\{j<k\colon\psi_e(j)\downarrow\}\cup\{k\}&\mbox{ if }|\{j<k\colon\psi_e(j)\downarrow\}|<m
\end{cases}
\]

One can easily check that $(\ast\mid L(e))$ belongs to the domain of ${\tt Error}_{m/k+1}$.
For an outer reduction $K$, define $K(e,j)=j$ for any $j<k$.
To compute $K(e,k)$ along computable ordinal steps, wait for finding $m$ many $j<k$ such that $\psi_e(j)\downarrow$.
If it is found at some ordinal stage, then $K(e,k)$ is defined as the least $\ell<k$ such that $\ell\not=j$ for any such $j<k$.
Otherwise, the computation never terminates, i.e., $K(e,k)\uparrow$.
One can easily see that $K$ is $\Pi^1_1$.

For readers who are not familiar with ordinal computability, we describe the details.
Let $\mathcal{O}\subseteq\N$ be Kleene's system of ordinal notations; see e.g.~\cite{sacks2,ChYuBook}.
As Kleene's $\mathcal{O}$ is $\Pi^1_1$-complete, and the graph $G_e$ of $\psi_e$ is $\Pi^1_1$, there exists a many-one reduction $p$ witnessing $G_e\leq_m\mathcal{O}$, where $\leq_m$ denotes many-one reducibility.
Then, for $a\in\mathcal{O}$, one can consider the stage $a$ approximation $\psi_e[a]$ of $\psi_a$; that is, $\psi_e(n)[a]\downarrow=m$ if and only if $p(n,m)<_\mathcal{O}a$.
Note that $<_\mathcal{O}$ is c.e.~(see e.g.~Sacks \cite[Theorem I.3.5]{sacks2}); hence $\psi_e[a]\downarrow$ is also a c.e.~property.
Then define $G_K$ as follows:
\begin{align*}
(e,k,\ell)\in G_K\iff(\exists a\in\mathcal{O})\big[&(\exists^{\geq m}j<k)\;\psi_e(j)[a]\downarrow\\
&\land\;(\forall j<k)\;[\psi_e(j)[a]\downarrow\;\rightarrow\;j\not=\ell<k]\big]
\end{align*}

One can easily see that $G_K$ is $\Pi^1_1$ since $\mathcal{O}$ is $\Pi^1_1$.
Hence, by $\Pi^1_1$-uniformization (see Sacks \cite[Theorem II.2.3]{sacks2} or Moschovakis \cite[4B.4]{MosBook}), there exists a partial $\Pi^1_1$ function $K\pcolon\N^2\to\N$ such that $G_K(e,k,K(e,k))$ holds whenever $(e,k,\ell)\in G_K$ for some $\ell$.
As in the proof of Proposition \ref{prop:llpo-vs-error-pos}, one can see that $\langle L,K\rangle$ witnesses $\Pi^1_1\mbox{-}{\tt LLPO}_{m/k}\leq_{hLT}^1{\tt Error}_{m/k+1}$.
\end{proof}

Now we introduce the notion of hyperarithmetical reducibility for bilayer functions.
\Art's hyperarithmetic strategy is a code $\tau$ for a partial $\Pi^1_1$ function $h_\tau\colon\N^{<\N}\to\N$.

\begin{definition}
Let $f$ and $g$ be bilayer functions.
We say that {\em $f$ is hyperarithmetically LT-reducible to $g$} (written $f\leq_{hLT}g$) if there exists a hyperarithmetical winning \Art-\Ni strategy for $\mathfrak{G}(f,g)$.
\end{definition}

The following is an analogue of Proposition \ref{prop:denerror-hyp}:

\begin{prop}\label{prop:denerror-hyp2}
Let ${\tt P}$ be a partial multifunction whose codomain is $\ell\in\N$ with $\ell>0$.
For any $\ep<1/(\ell+1)$, if ${\tt P}\leq_{hLT}{\tt DenError}_\ep$, then ${\tt P}$ is hyperarithmetic.
\end{prop}

\begin{proof}
The argument is the same as Proposition \ref{prop:denerror-hyp}.
Only the complexity of $Q$ needs to be considered.
If we consider a hyperarithmetical strategy, $\Phi_\tau^n$ is no longer a computable function, but a $\Pi^1_1$ function.
For this reason, $\Psi$ is also $\Pi^1_1$.
To see this, note that $\Psi^g(n)$ is defined to be $j$ if and only if there exists a $g$-fat finite tree $S\subseteq T^{g}$ such that $\Phi_\tau^n$ is defined and constant on the leaves of $S$ and its unique value is $j$.
This condition is $\Pi^1_1$ since $S$ is finite and $\Phi_\tau^n$ is $\Pi^1_1$.
Moreover, as mentioned in the proof of Proposition \ref{prop:denerror-hyp}, if $\Psi^g(n)$ is defined then the algorithm $\Psi$ only reads $g$ up to the height of $S$.
Thus, $Q(n,j)$ holds if and only if for any $f$, there exists a finite string $\sigma$ such that $\sigma$ majorizes $f$ up to $|\sigma|$ and $\Psi^\sigma(n)\downarrow=j$.
This condition is $\Pi^1_1$.
The rest follows the same argument as in Proposition \ref{prop:denerror-hyp}.
\end{proof}

\begin{cor}
For any $\ell\geq 2$, ${\tt Error}_{1/\ell}\leqLT{\tt DenError}_{\ep}$ if and only if $1/\ell\leq\ep$.
\end{cor}

\begin{proof}
The backward direction follows from Observation \ref{obs:err-vs-denerr}.
For the forward direction, assume that $1/\ell>\ep$.
It is clear that $\Pi^1_1\mbox{-}{\tt LLPO}_{1/(\ell-1)}$ is not hyperarithmetic.
In particular, by Proposition \ref{prop:denerror-hyp2}, we have $\Pi^1_1\mbox{-}{\tt LLPO}_{1/(\ell-1)}\not\leq_{hLT}{\tt DenError}_{\ep}$ since $\ep<1/\ell$.
However, by Proposition \ref{prop:pillpo-vs-error-pos}, we have $\Pi^1_1\mbox{-}{\tt LLPO}_{1/(\ell-1)}\leq_{hLT}{\tt Error}_{1/\ell}$. 
Hence, ${\tt Error}_{1/\ell}\not\leq_{hLT}{\tt DenError}_{\ep}$.
\end{proof}

Note that the above proof also shows that $\Pi^1_1\mbox{-}{\tt LLPO}_{1/\ell}\leq_{hLT}{\tt DenError}_{\ep}$ if and only if $1/(\ell+1)\leq\ep$.

The above results say nothing about ${\tt DenError}_0$.
Note that since the asymptotic density of a cofinite set is $1$, we have ${\tt Cofinite}\leqoLT{\tt DenError}_0$.
We show that computability with density error $0$ is strictly stronger than computability with finitely many error in the following sense:

\begin{theorem}
${\tt Cofinite}\lLT{\tt DenError}_0$.
\end{theorem}

\begin{proof}
Suppose not.
Then, there exists a winning \Art-\Ni strategy $(\tau\mid\eta)$ witnessing ${\tt DenError}_0\leqLT{\tt Cofinite}$.
Except for the first move $(\ast\mid A)$, \Mer's move is always a number $j\in\N$, which yields the tree $\N^{<\N}$ of all possible moves by \Mer.
Here $A$ is a secret input, which is invisible to \Art.
Hence, \Art's strategy $\tau$ yields a partial computable function $\Phi_\tau\pcolon\N^{<\N }\to\N $, where $\Phi_\tau(\sigma)\downarrow=u$ if and only if, after reading \Mer's moves $\sigma$, \Art's strategy $\tau$ declares termination with $u$.
For $h\colon\N^{<\N}\to\N$, consider the tree $\N^{<\N}[\geq h]=\{\sigma\in\N^{<\N}:(\forall n<|\sigma|)\;h(\sigma\upto n)\leq\sigma(n)\}$.
\Nim's strategy $\eta$ restricts \Mer's possible moves to the tree $\N^{<\N}[\geq \eta_A]$, where $\eta_A(\sigma)=\eta(A\fr\sigma)$, and moreover, as $\tau$ is winning, the computation $\Phi_\tau$ determines covers the tree $\N^{<\N}[\geq \eta_A]$; that is, for any infinite path $x$ through $\N^{<\N}[\geq \eta_A]$ there exists an initial segment $\xi$ of $x$ such that $\Phi_\tau(\xi)$ is defined.

This ensures the existence of a function $h\colon\N^{<\N}\to\N$ such that the computation $\Phi_\tau$ covers the tree $\N^{<\N}[\geq h]$.
Consider the set $B$ of all minimal strings $\xi\in\N^{<\N}[\geq h]$ such that $\Phi_\tau(\xi)$ is defined.
Note that any infinite path $x$ through $\N^{<\N}[\geq h]$ has an initial segment in $B$.
Then, $B$ yields a well-founded subtree $T=\{\zeta\in\N^{<\N}:(\exists\xi\in B)\;\zeta\preceq\xi\}$ of $\N^{<\N}[\geq h]$ so that $B$ is the set of all leaves of $T$.
One can see that if $\sigma\in T$ is not a leaf then the set ${\rm succ}_T(\sigma)$ of its immediate successors is cofinite.
This is because, for any $n\geq h(\sigma)$, any infinite path extends $\sigma\fr n\in \N^{<\N}[\geq h]$ has an initial segment $\rho\in B$.
Then $\rho$ is a leaf of $T$, and since $\sigma \in T$ is not a leaf, $\rho$ must extend $\sigma\fr n$.
Hence we have $\sigma\fr n\in T$ since a tree is $\preceq$-downward closed.

We label each node of this well-founded tree as follows:
First, a leaf $\rho\in T$ is labeled by the value of $\Phi_\tau(\rho)$.
If $\sigma\in T$ is not a leaf, then turn to its immediate successors.
If $\sigma$ has infinitely many immediate successors which have the same label $c$, then $\sigma$ is also labeled by $c$.
If there is no such label $c$, then $\sigma$ is labeled by $\infty$.

Now, suppose that the label of the root of $T$ is $c\not=\infty$.
Then, \Me plays $\N\setminus\{c\}$ as his first move, which has clearly asymptotic density $1$.
In the following, we assume that \Ar and \Ni follows their winning strategies $\tau$ and $\eta$, respectively.
If \Ni reacts to the above move with $z_0$, search for $x_1\geq z_0$ such that $\langle x_1\rangle\in T$ and the label of $\langle x_1\rangle$ is $c$.
Such an $x_1$ exists, since for the former condition $\langle x_1\rangle\in T$, recall that the set of immediate successors of a node in $T$ is cofinite, and for the latter condition, the label of the root is $c$, so there are infinitely many immediate successors labeled by $c$.
Then \Me plays $x_1$ as his next move.
Continuing this argument, \Me can keep returning nodes of $T$ with the same label $c$, and \Mer's moves eventually reach a leaf of $T$.
Reaching a leaf means that \Ar declares termination of the game with some value $\Phi_\tau(\rho)$, but since the label of this leaf $\rho$ is $c$, the value $\Phi_\tau(\rho)$ must be $c$.
Since $c\not\in{\tt DenError}_0(\N\setminus\{c\})$ and \Mer's first move is $\N\setminus\{c\}$, this means that \Me wins the game, which contradicts our assumption that $(\tau\mid\eta)$ is a winning \Art-\Ni strategy.

Thus, the root of $T$ must be labeled by $\infty$.
We say that a node $\sigma'\in T$ is a big sibling of a node $\sigma\in T$ if $\sigma$ and $\sigma'$ take the same value except for the last entry, and $\sigma'$ is larger than $\sigma$ for the last entry; that is, $\sigma(n)=\sigma'(n)$ for any $n<|\sigma|-1$ and $\sigma(|\sigma|-1)<\sigma'(|\sigma|-1)$. 
We also say that a node $\sigma\in T$ is decisive if all proper initial segments of $\sigma$ are labeled by $\infty$, but $\sigma$ and all big siblings of $\sigma$ are labeled by some values in $\N$.
Note that the root of $T$ is not decisive as it is labeled by $\infty$, so any decisive node has a proper initial segment.
Let $(\alpha_s)_{s\in\N}$ be a list of all decisive nodes of $T$.
First put $d_0=1$.
At stage $s$, assume that $d_s$ has already been defined.
The immediate predecessor $\alpha_s^-$ of $\alpha_s$ is labeled by $\infty$ since $\alpha_s$ is decisive.
The label $\infty$ of $\alpha_s^-$ means that, for any $c\in\N$, there are only finitely many siblings of $\alpha_s$ labeled by $c$.
Therefore, by the pigeonhole principle, $\alpha_s$ has a big sibling labeled by some $c_s>d_s$.
As $\alpha_s$ is decisive, $c_s$ must be a finite value.
Then, put $d_{s+1}=2c_s$.

Now \Me plays $\N\setminus\{c_s:s\in\N\}$ as his first move, which has asymptotic density $1$ since $c_{s+1}>2c_s$ for any $s$ by our construction.
At the round $n+1$, assume that the history of \Mer's previous moves is $x_1,\dots,x_n$, and \Nim's previous move is $z_n$.
First consider the case that the label of $\sigma_n=\langle x_1,\dots,x_n\rangle$ is $\infty$.
If $\sigma_n$ has infinitely many immediate successors labeled by $\infty$, then as his next move \Me plays $x_{n+1}\geq z_n$ so that $\sigma_n\fr x_{n+1}\in T$ is labeled by $\infty$.
Otherwise, there are only finitely many immediate successors of $\sigma_n$ labeled by $\infty$, and thus, there exists a decisive immediate successor of $\sigma_n$ of the form $\sigma_n{}\fr z$ for some $z\geq z_n$.
Then we must have $\alpha_s=\sigma_n{}\fr z$ for some $s$.
As seen above, $\alpha_s$ has a big sibling $\alpha_s'=\sigma_n{}\fr z'$ labeled by $c_s$.
As his next move, \Me plays the last entry $x_{n+1}:=z'$ of $\alpha_s'$.
Note that the history $\alpha_s'=\langle x_1,\dots,x_n,x_{n+1}\rangle$ of moves is now labeled by $c_s$.
Next consider the case that the label of $\sigma_n=\langle x_1,\dots,x_n\rangle$ has already become a finite value $c\in\N$.
In this case, $\sigma_n$ has infinitely many immediate successors labeled by $c$, and then as his next move \Me can play $x_{n+1}\geq z_n$ so that $\sigma_n\fr x_{n+1}\in T$ is labeled by $c$.

As this play follows a winning \Art-\Ni strategy, \Ar declares termination at some round.
Then, the history of \Mer's moves eventually reaches a leaf of $T$ which is labeled by a finite value.
Hence, the history of \Mer's moves is labeled by a finite value at some round, and once the label becomes a finite value, our construction of \Mer's strategy ensures that the value of the label does not change after that.
Indeed, \Mer's strategy described above stabilizes the labels of the histories of \Mer's moves to $c_s$ for some $s$.
Therefore, the history of \Mer's moves eventually reaches a leaf of $T$ which is labeled by $c_s$, which turns out to be (the second coordinate of) \Art's last move since the leaf of $T$ is labeled by the value of $\Phi_\tau$ on it.
However, $c_s\not\in{\tt DenError}_0(\N\setminus\{c_s:s\in\N\})$ and \Mer's first move is $\N\setminus\{c_s:s\in\N\}$.
This means that \Me wins the game, which contradicts our assumption that $(\tau\mid\eta)$ is a winning \Art-\Ni strategy.
Consequently, there exists no winning \Art-\Ni strategy; hence ${\tt Cofinite}\lLT{\tt DenError}_0$.
\end{proof}

Hence, we get the strict hierarchy of computability with error density $\ep$:
\[{\tt Cofinite}\lLT{\tt DenError}_0\lLT\dots\lLT{\tt DenError}_{1/3}\lLT{\tt DenError}_{1/2}\eqLT\neg\neg.\]

\begin{figure}[t]
{\footnotesize
\includegraphics[width=9.5cm]{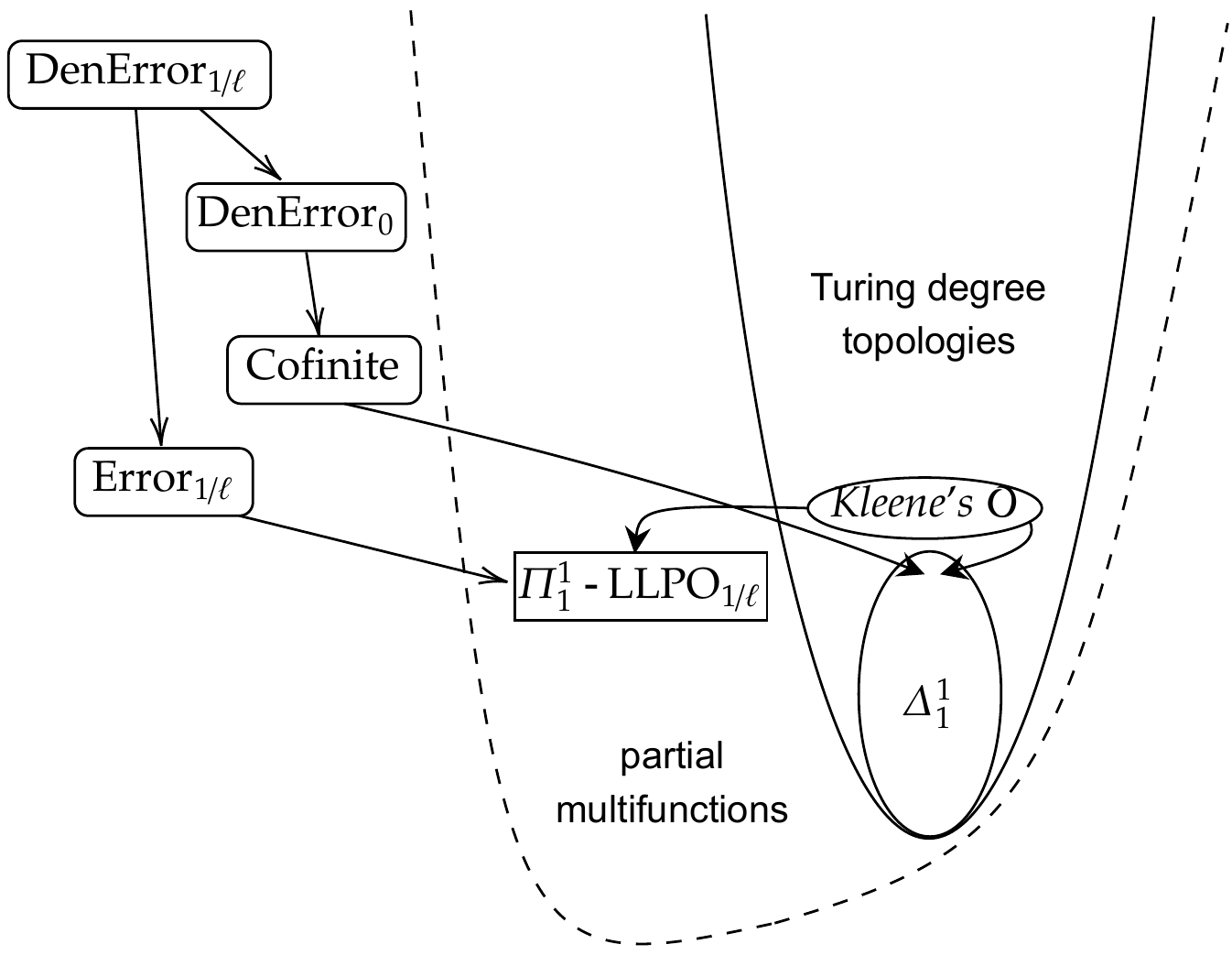}
\caption{\footnotesize Higher parts on Lawvere-Tierney topologies on the effective topos}\label{figure:topology2}
}
\end{figure}

Figure \ref{figure:topology2} summarizes some of basic implications about the $\leq_{\rm rea}$-ordering on the Lawvere-Tierney topologies (around hyperarithmetic Turing topologies) on the effective topos, where $A\to B$ means $B\leqLT A$ (or $B^{\Game\to}\leq_{\rm rea}A^{\Game\to}$).

\section{Future work}

%Proposition \ref{prop:error-vs-proberror} can be improved so that ${\tt LLPO}_{1/\ell}^{\emptyset'}$ is not ``$k$-query Turing reducible'' to ${\tt ProbError}_\ep$ if and only if $\frac{1}{\ell}\leq\ep$, for any $k\in\N$.
%However, the $\N$-version of ${\tt LLPO}$ lacks compactness, so it does not seem possible to compute an upper bound $k$ on the number of times of making queries.
%So, we propose the following:
%
%\begin{question}
%Does there exist an oracle $\alpha$ such that ${\tt LLPO}_{1/\ell}^\alpha\leq_T{\tt ProbError}_\ep$ if and only if $\frac{1}{\ell}\leq\ep$?
%\end{question}
%
%By Proposition \ref{prop:llpo-vs-proberror}, one can deduce that ${\tt ProbError}_{1/(\ell+1)}<_T{\tt ProbError}_{1/\ell}$.
%Moreover, by the same argument as in \cite[Theorem 10.1]{BGH15}, one can shows that $\ep<\ep'$ implies ${\tt ProbError}_{\ep}<_{1T}{\tt ProbError}_{\ep'}$.
%However, for Turing reducibility, as can be seen from Theorem \ref{thm:error-separation}, the situation does not seem to be so simple.
%In particular, the following is left open.
%
%
%\begin{question}
%If $\ep<\ep'$ does ${\tt ProbError}_{\ep}<_T{\tt ProbError}_{\ep'}$ always hold?
%\end{question}

%We do not know if Pitts' topology is strictly below that of computability with density zero errors.
%
%\begin{question}
%Does ${\tt Cofinite}\lLT{\tt DenError}_0$ hold?
%\end{question}

One may come up with other basic bilayer functions not mentioned in this article, but we do not know which ones are non-trivial and interesting.
It is a vague question, but finding interesting basic bilayer functions is a big problem in itself.

\begin{question}
Is there any other interesting basic bilayer function?
\end{question}

In Section \ref{sec:prob-comp-er}, we have seen that the $\N$-version of weak weak K\"onig's lemma, ${\tt WWKL}$, is Turing equivalent to ${\tt LLPO}$.
Due to this kind of phenomenon, unlike $\N^\N$-computation, it is difficult to find a nontrivial partial multifunction in the context of $\N$-computation.
There are partial multifunctions on $\N$ not mentioned so far, such as {\em all-or-unique choice} ${\tt AoUC}_X$ on $X$ (see e.g.~\cite{KiPa16}).
However, the $\N$-version of ${\tt AoUC}_{2^\N}$ turns out to be Turing equivalent to ${\tt LLPO}$ by the same argument as above, and ${\tt AoUC}_\N$ is Turing equivalent to the halting problem by considering enumeration time functions as in Section \ref{sec:perfect-information-game}.

\begin{question}
Is there any other natural partial multifunction on $\N$ whose Turing degree strictly lie between the computable ones and the halting problem?
\end{question}

In this article, we have focused on topologies on the effective topos; however ``{\it the world of computable mathematics}'' for modern computability theorists seems to be the Kleene-Vesley topos rather than the effective topos.
As indicated in Section \ref{sec:imperfect-information-game}, the structure of Lawvere-Tierney topologies on the Kleene-Vesley topos seems isomorphic to the bilayer version of generalized Weihrauch reducibility.
This structure should also be explored in depth in the future.

As another topos, the realizability topos ${\bf RT}(K_2)$ induced by Kleene's second algebra corresponds to ``{\it the world of continuous mathematics},'' and a Lawvere-Tierney topology on the topos is a kind of data that indicates how much discontinuity to add to the world.
One can see that the structure of Lawvere-Tierney topologies on the topos ${\bf RT}(K_2)$ is isomorphic to the bilayer version of generalized continuous Weihrauch reducibility.

In general, there are many other toposes that are related to computability theory and (effective) descriptive set theory.
As mentioned in Kihara \cite{Kih20}, any $\Sigma^\ast$-pointclass (see \cite{MosBook}) yields a (relative) partial combinatory algebra, which induces a topos.
If the pointclass $\tpbf{\Pi}^1_1$ is used as a seed, a topos corresponding to ``the world of Borel measurable mathematics'' will be created, and if the pair $(\Pi^1_1,\tpbf{\Pi}^1_1)$ is used, a topos corresponding to ``the world of effective Borel measurable mathematics'' will be created.
These lead us to the study of ``Lawvere-Tierney topologies for (effective) descriptive set theorists.''
The topologies in this case are some sort of data that indicate how much non-Borel objects to add to the world.

It may also be reasonable to study these structures in the context of synthetic descriptive set theory \cite{PaBr15}.
We leave the exploration of these structures as a future task.

\begin{ack}
The author would like to thank Satoshi Nakata for valuable discussions.
Kihara's research was partially supported by JSPS KAKENHI Grant 19K03602 and 21H03392, and the JSPS-RFBR Bilateral Joint Research Project JPJSBP120204809.
\end{ack}

\bibliographystyle{plain}
\bibliography{references}

\end{document}